\documentclass[12pt]{amsart}
\newcommand{\deleted}[1]{}
\newcommand{\delete}[1]{}
\newcommand{\mynotes}[1]{}
\newcommand\notes[1]{}
\usepackage{txfonts}
\usepackage{amssymb}
\usepackage{eucal}
\usepackage{graphicx}
\usepackage{amsmath}
\usepackage{amscd}
\usepackage[all]{xy}           

\usepackage{amsfonts,latexsym}
\usepackage{xspace}
\usepackage{epsfig}
\usepackage{float}
\usepackage{color}
\usepackage{fancybox}
\usepackage{colordvi}
\usepackage{multicol}
\usepackage{tikz}

\usepackage{wasysym}
\usepackage{xspace}
\usepackage[colorlinks,final,backref=page,hyperindex,hypertex]{hyperref}

\usepackage[active]{srcltx} 

\usepackage{graphicx}
\renewcommand\arraystretch{1.5}
\newcommand\changed[1]{#1}
\changed{}

\newtheorem{theorem}{Theorem}[section]
\newtheorem{lemma}[theorem]{Lemma}
\newtheorem{coro}[theorem]{Corollary}

\newtheorem{prop}[theorem]{Proposition}
\theoremstyle{definition}
\newtheorem{defn}[theorem]{Definition}
\newtheorem{remark}[theorem]{Remark}
\newtheorem{exam}[theorem]{Example}

\deleted{\newtheorem{theorem}{Theorem}[section]
\newtheorem{prop-def}{Proposition-Definition}[section]
\newtheorem{coro-def}{Corollary-Definition}[section]

}

\newcommand\yy[2]{\begin{scope}[scale=1/2^#1]
\draw (#2,0)--+(1/2,-1/2)--+(1,0);
\end{scope}}

\newcommand\YY[2][]{%
\tikz[line width=0.18ex,scale=0.75,baseline=-3ex,inner sep=1pt,#1]{
\draw (0,0)--+(1/2,-1/2)--+(1,0) (1/2,-1/2)--+(0,-1/2); #2}}

\newcommand{\tquatredeuxa}{\begin{picture}(15,18)(-5,-1)
\put(3,0){\circle*{2}}
\put(-0.2,0.2){$\vee$}
\put(6,7){\circle*{2}}
\put(0,7){\circle*{2}}
\put(0,14){\circle*{2}}
\put(0,7){\line(0,1){7}}
\put(-3,4){\line(1,0){7}}
\put(5,-2){\tiny $\alpha$}
\put(9,5){\tiny $\gamma$}
\put(-6,5){\tiny $\beta$}
\put(-6,12){\tiny $\delta$}
\end{picture}}

\newcommand{\tddeux}[2]{\begin{picture}(12,5)(0,-1)
\put(3,0){\circle*{2}}
\put(3,0){\line(0,1){5}}
\put(3,5){\circle*{2}}
\put(6,-3){\tiny #1}
\put(6,3){\tiny #2}
\end{picture}}

\newcommand{\nc}{\newcommand}
\nc{\tred}[1]{\textcolor{red}{#1}} \nc{\tblue}[1]{\textcolor{blue}{#1}} \nc{\tgreen}[1]{\textcolor{green}{#1}} \nc{\tpurple}[1]{\textcolor{purple}{#1}} \nc{\btred}[1]{\textcolor{red}{\bf #1}} \nc{\btblue}[1]{\textcolor{blue}{\bf #1}} \nc{\btgreen}[1]{\textcolor{green}{\bf #1}} \nc{\btpurple}[1]{\textcolor{purple}{\bf #1}}

\renewcommand{\Bbb}{\mathbb}

\newcommand{\efootnote}[1]{}

\newcommand\wyscco[1]{}
\renewcommand{\textbf}[1]{}

\nc{\mlabel}[1]{\label{#1}}  
\nc{\mcite}[1]{\cite{#1}}  
\nc{\mref}[1]{\ref{#1}}  
\nc{\mbibitem}[1]{\bibitem{#1}} 

\delete{
\nc{\mlabel}[1]{\label{#1}{\hfill \hspace{1cm}{\bf{{\ }\hfill(#1)}}}}
\nc{\mcite}[1]{\cite{#1}{{\bf{{\ }(#1)}}}}  
\nc{\mref}[1]{\ref{#1}{{\bf{{\ }(#1)}}}}  
\nc{\mbibitem}[1]{\bibitem[\bf #1]{#1}} 
}

\renewcommand\geq{\geqslant}

\renewcommand\bar[1]{\overline{#1}}
\renewcommand\tilde[1]{\widetilde{#1}}

\nc\kdot{\bfk}
\nc\simple{simple\xspace}

\nc{\rbw}{\mathfrak{R}} \nc{\brp}{\mathrm{brp}} \nc{\lead}{\mathrm{Lead}} \nc{\Id}{\mathrm{Id}} \nc{\Irr}{\mathrm{Irr}} \nc{\vx}{\sigma} \nc{\vy}{\tau} \nc{\dvx}{\sigma^{(1)}} \nc{\dvy}{\tau^{(1)}} \nc{\done}{\vep} \nc{\citep}[1]{\cite{#1}} \nc{\wt}{\mathrm{wt}} \nc{\bre}[1]{|#1|} \nc{\mapmonoid}{\frakM} \nc{\disjoint}{\frakM'}
\nc{\ncpoly}[1]{\langle #1\rangle}  
\nc{\mapm}[1]{\frakM(#1)}
\nc{\diff}[1]{{}^\NC\{ #1 \}} \nc{\disj}[1]{\{{#1}\}'} \nc{\mdisj}[1]{\frakM'(#1)} \nc{\brho}{\bar{\rho}} \nc{\om}{\bar{\frakm}} \nc{\frakn}{\mathfrak n} \nc{\ddeg}[1]{^{(#1)}} \nc{\opset}{X} \nc{\genset}{{Z}} \nc{\NC}{\mathrm{{NC}}} \nc{\leaf}{\mathrm{leaf}} \nc{\twig}{\mathrm{twig}} \nc{\fe}{\mathrm{fl}} \nc{\munderline}[1]{#1} \nc{\bo}{o} \nc{\dep}{\mathrm{dep}} \nc{\ofe}{\mathrm{ofl}} \nc{\dfe}{\mathrm{dfe}} \nc{\fex}{\mathrm{fex}} \nc{\dl}{\mathrm{dlex}} \nc{\db}{\mathrm{db}} \nc{\lex}{\mathrm{lex}} \nc{\clex}{\mathrm{clex}} \nc{\dgp}{\mathrm{dgp}} \nc{\dgx}{\mathrm{dgx}} \nc{\br}{\mathrm{br}} \nc{\obd}{\mathrm{odb}} \nc{\ob}{\mathrm{ob}}
\nc{\loc}{location\xspace}
\nc{\occ}{occurrence\xspace}
\nc{\occs}{occurrences\xspace}
\nc{\pla}{placement\xspace}
\nc{\plas}{placements\xspace}
\nc{\tx}{\tilde{X}}

\nc{\bin}[2]{ (_{\stackrel{\scs{#1}}{\scs{#2}}})}  
\nc{\binc}[2]{ \left (\!\! \begin{array}{c} \scs{#1}\\
    \scs{#2} \end{array}\!\! \right )}  
\nc{\bincc}[2]{  \left ( {\scs{#1} \atop
    \vspace{-1cm}\scs{#2}} \right )}  
\nc{\bs}{\bar{S}} \nc{\cosum}{\sqsubset} \nc{\la}{\longrightarrow} \nc{\rar}{\rightarrow} \nc{\dar}{\downarrow} \nc{\dprod}{**} \nc{\dap}[1]{\downarrow \rlap{$\scriptstyle{#1}$}} \nc{\md}{\mathrm{dth}} \nc{\uap}[1]{\uparrow \rlap{$\scriptstyle{#1}$}} \nc{\defeq}{\stackrel{\rm def}{=}} \nc{\disp}[1]{\displaystyle{#1}} \nc{\dotcup}{\ \displaystyle{\bigcup^\bullet}\ } \nc{\gzeta}{\bar{\zeta}} \nc{\hcm}{\ \hat{,}\ } \nc{\hts}{\hat{\otimes}} \nc{\barot}{{\otimes}} \nc{\free}[1]{{#1}^\ast} \nc{\uni}[1]{\tilde{#1}} \nc{\hcirc}{\hat{\circ}} \nc{\leng}{\ell} \nc{\lleft}{[} \nc{\lright}{]} \nc{\lc}{\lfloor} \nc{\rc}{\rfloor}
\nc{\lb}{[} 
\nc{\rb}{]} 
\nc{\curlyl}{\left \{ \begin{array}{c} {} \\ {} \end{array}
    \right.  \!\!\!\!\!\!\!}
\nc{\curlyr}{ \!\!\!\!\!\!\!
    \left. \begin{array}{c} {} \\ {} \end{array}
    \right \} }
\nc{\longmid}{\left | \begin{array}{c} {} \\ {} \end{array}
    \right. \!\!\!\!\!\!\!}
\nc{\onetree}{\bullet} \nc{\ora}[1]{\stackrel{#1}{\rar}}
\nc{\ola}[1]{\stackrel{#1}{\la}}
\nc{\ot}{\otimes} \nc{\mot}{{{\boxtimes\,}}} \nc{\otm}{\overline{\boxtimes}} \nc{\sprod}{\bullet} \nc{\scs}[1]{\scriptstyle{#1}} \nc{\mrm}[1]{{\rm #1}} \nc{\msum}{\sum\limits}
\nc{\margin}[1]{\marginpar{\rm #1}}   
\nc{\dirlim}{\displaystyle{\lim_{\longrightarrow}}\,} \nc{\invlim}{\displaystyle{\lim_{\longleftarrow}}\,} \nc{\mvp}{\vspace{0.3cm}} \nc{\tk}{^{(k)}} \nc{\tp}{^\prime} \nc{\ttp}{^{\prime\prime}} \nc{\svp}{\vspace{2cm}} \nc{\vp}{\vspace{8cm}} \nc{\proofbegin}{\noindent{\bf Proof: }}
\nc{\proofend}{$\blacksquare$ \vspace{0.3cm}}
\nc{\modg}[1]{\!<\!\!{#1}\!\!>}
\nc{\intg}[1]{F_C(#1)} \nc{\lmodg}{\!<\!\!} \nc{\rmodg}{\!\!>\!} \nc{\cpi}{\widehat{\Pi}}
\nc{\sha}{{\mbox{\cyr X}}}  
\nc{\shap}{{\mbox{\cyrs X}}} 
\nc{\shpr}{\diamond}    
\nc{\shp}{\ast} \nc{\shplus}{\shpr^+}
\nc{\shprc}{\shpr_c}    
\nc{\msh}{\ast} \nc{\zprod}{m_0} \nc{\oprod}{m_1} \nc{\vep}{\varepsilon} \nc{\labs}{\mid\!} \nc{\rabs}{\!\mid}
\nc{\astarrow}{\overset{\raisebox{-2pt}{{\scriptsize $\ast$}}}{\rightarrow}}
\nc{\astlarrow}{\overset{\raisebox{-2pt}{{\scriptsize $\ast$}}}{\longrightarrow}}
\nc{\lastarrow}{\overset{\raisebox{-2pt}{{\scriptsize $\ast$}}}{\leftarrow}}
\nc{\mastarrow}[1]{\overset{\raisebox{-2pt}{{\scriptsize $#1$}}}{\rightarrow}}
\nc{\quvarrow}[3]{#1 \overset{q,u,v}{\longrightarrow}_{#3} #2}
\nc{\quvkto}[1]{f_{#1} \overset{q_{#1}, u_{#1}, v_{#1}}{\longrightarrow}_\phi g_{#1}}
\nc{\tvarrow}[3]{#1 \overset{(t,v)}{\longrightarrow}_{#3} #2}
\nc{\Supp}{{\rm Supp}}
\nc{\mpu}{u^{\ast}}
\nc{\mpv}{v^{\ast}}
\nc{\mpw}{w^{\ast}}
\nc{\mpx}{x^{\ast}}
\nc{\dps}{\dotplus}

\nc{\dth}{d} \nc{\mmbox}[1]{\mbox{\ #1\ }} \nc{\fp}{\mrm{FP}} \nc{\rchar}{\mrm{char}} \nc{\Fil}{\mrm{Fil}} \nc{\Mor}{Mor\xspace} \nc{\gmzvs}{gMZV\xspace} \nc{\gmzv}{gMZV\xspace} \nc{\mzv}{MZV\xspace} \nc{\mzvs}{MZVs\xspace} \nc{\Hom}{\mrm{Hom}} \nc{\id}{\mrm{id}} \nc{\im}{\mrm{im}} \nc{\incl}{\mrm{incl}} \nc{\map}{\mrm{Map}} \nc{\mchar}{\rm char} \nc{\nz}{\rm NZ} \nc{\supp}{\mathrm Supp}
\nc{\mo}{\mathbf o}
\nc{\pl}{\mathfrak{p}}

\nc{\Alg}{\mathbf{Alg}} \nc{\Bax}{\mathbf{Bax}} \nc{\bff}{\mathbf f} \nc{\bfk}{{\bf k}} \nc{\bfone}{{\bf 1}} \nc{\bfx}{\mathbf x} \nc{\bfy}{\mathbf y}
\nc{\base}[1]{\bfone^{\otimes ({#1}+1)}} 
\nc{\Cat}{\mathbf{Cat}} \delete{}
\nc{\detail}{\marginpar{\bf More detail}
    \noindent{\bf Need more detail!}
    \svp}
\nc{\Int}{\mathbf{Int}} \nc{\Mon}{\mathbf{Mon}}
\nc{\rbtm}{{shuffle }} \nc{\rbto}{{Rota-Baxter }} \nc{\remarks}{\noindent{\bf Remarks: }} \nc{\Rings}{\mathbf{Rings}} \nc{\Sets}{\mathbf{Sets}}
\nc{\vwpt}{{Let $V$ be a free $\bfk$-module with a $\bfk$-basis $W$ and let $\Pi$ be a \simple term-rewriting system on $V$ with respect to $W$.}\xspace}

\nc{\BA}{{\Bbb A}} \nc{\CC}{{\Bbb C}} \nc{\DD}{{\Bbb D}} \nc{\EE}{{\Bbb E}} \nc{\FF}{{\Bbb F}} \nc{\GG}{{\Bbb G}} \nc{\HH}{{\Bbb H}} \nc{\LL}{{\Bbb L}} \nc{\NN}{{\Bbb N}} \nc{\KK}{{\Bbb K}} \nc{\QQ}{{\Bbb Q}} \nc{\RR}{{\Bbb R}} \nc{\TT}{{\Bbb T}} \nc{\VV}{{\Bbb V}} \nc{\ZZ}{{\Bbb Z}}

\nc{\cala}{{\mathcal A}} \nc{\calc}{{\mathcal C}} \nc{\cald}{{\mathcal D}} \nc{\cale}{{\mathcal E}} \nc{\calf}{{\mathcal F}} \nc{\calg}{{\mathcal G}} \nc{\calh}{{\mathcal H}} \nc{\cali}{{\mathcal I}} \nc{\call}{{\mathcal L}} \nc{\calm}{{\mathcal M}} \nc{\caln}{{\mathcal N}} \nc{\calo}{{\mathcal O}} \nc{\calp}{{\mathcal P}} \nc{\calr}{{\mathcal R}} \nc{\cals}{{\mathcal S}} \nc{\calt}{{\mathcal T}} \nc{\calw}{{\mathcal W}}
\nc{\calv}{{\mathcal V}}
\nc{\calk}{{\mathcal K}} \nc{\calx}{{\mathcal X}} \nc{\CA}{\mathcal{A}}

\nc{\fraka}{{\mathfrak a}} \nc{\frakA}{{\mathfrak A}} \nc{\frakb}{{\mathfrak b}} \nc{\frakB}{{\mathfrak B}} \nc{\frakC}{{\mathfrak C}}
\nc{\frakD}{{\mathfrak D}} \nc{\frakH}{{\mathfrak H}} \nc{\frakM}{{\mathfrak M}} \nc{\bfrakM}{\overline{\frakM}} \nc{\frakm}{{\mathfrak m}} \nc{\frkP}{{\mathfrak P}}
\nc{\frakN}{{\mathfrak N}} \nc{\frakp}{{\mathfrak p}} \nc{\fraku}{{\mathfrak u}} \nc{\frakv}{{\mathfrak v}}
\nc{\frakQ}{{\mathfrak Q}}\nc{\frakR}{{\mathfrak R}} \nc{\frakS}{{\mathfrak S}}
\nc{\frakx}{{\mathfrak x}} \nc{\ox}{\bar{\frakx}} \nc{\frakX}{{\mathfrak X}} \nc{\fraky}{{\mathfrak y}}
\nc\dop{\delta}
\nc{\Reduce}{{\rm Red}}

\font\cyr=wncyr10 \font\cyrs=wncyr7
\nc{\redt}[1]{\textcolor{red}{#1}}

\nc{\yi}[1]{\textcolor{red}{Yi:#1}} 
\nc{\lio}[1]{}
\nc{\sz}[1]{\textcolor{green}{sz:#1}}
\nc{\szo}[1]{}
\nc{\xing}[1]{\textcolor{purple}{Xing:#1}}
\nc{\ws}[1]{\textcolor{blue}{{#1}}} 
\nc{\wsc}[1]{\textcolor{blue}{ws:#1}} 
\nc{\wsco}[1]{}
\nc{\wsn}[1]{\textcolor{magenta}{#1}} 

\nc{\medmid}{{\,~{\tiny \longmid}~\,}}
 \nc{\lbar}[1]{\overline{#1}}
\nc{\sy}[1]{Y_{#1}^{\sigma}(\tilde{X})} \nc{\hlrs}{H_{LR}^\sigma(\tilde{X})}
\nc{\y}[2]{Y_{#1}(#2)} \nc{\vees}{\vee_\sigma}

\nc{\brw}{\frakM(Z)} \nc{\irr}{{\rm Irr}} \nc{\pis}{\Pi_S}
\nc{\term}{term\xspace} \nc{\Term}{Term\xspace}
\nc{\re}[1]{R(#1)} \nc{\sumre}[2]{R^{#1}_{#2}}
\nc{\stars}[2]{#1|_{#2}}
\nc{\nbfk}{\bfk^{\times}} \nc{\revise}[1]{\textcolor{blue}{#1}}
\nc{\ord}{{\rm ord}} \nc{\tpi}{\rightarrow_{\Pi}} \nc{\tpis}{\rightarrow_{\Pi_S}}
\nc{\fix}[1]{\tilde{#1}} \nc{\ars}{term rewriting system } \nc{\proj}{\mathrm{P_m}}
\nc{\topi}{\rightarrow_{\Pi_{S}}} \nc{\adm}{\text {\rm Adm}} \nc{\X}{\Omega}
\nc{\inte}{\text {\rm Int}}

\begin{document}
\title[Hopf algebras of decorated planar binary trees and $\vee$-cocycles]{Hopf algebras of planar binary trees: An operated algebra approach}
\author{Yi Zhang} \address{School of Mathematics and Statistics,
Lanzhou University, Lanzhou, Gansu 730000, P. R. China}
\email{zhangy2016@lzu.edu.cn}

\author{Xing Gao$^{*}$
}
\footnotetext{* Corresponding author.}
\address{School of Mathematics and Statistics,
Key Laboratory of Applied Mathematics and Complex Systems,
Lanzhou University, Lanzhou, Gansu 730000, P. R. China}
\email{gaoxing@lzu.edu.cn}

%

\hyphenpenalty=8000
\date{\today}

\begin{abstract}
Parallel to operated algebras built on top of planar rooted trees via the grafting operator $B^+$,
we introduce and study $\vee$-algebras and more generally $\vee_\X$-algebras based on planar binary trees.
Involving an analogy of the Hochschild 1-cocycle condition, cocycle $\vee_\X$-bialgebras (resp.~$\vee_\X$-Hopf algebras)
are also introduced and their free objects are constructed via decorated planar binary trees.
As a special case, the well-known Loday-Ronco Hopf algebra $H_{LR}$ is a free cocycle $\vee$-Hopf algebra.
By means of admissible cuts, a combinatorial description of the coproduct $\Delta_{LR(\X)}$  on decorated planar binary trees is given, as in the Connes-Kreimer Hopf algebra by admissible cuts.
\end{abstract}

\subjclass[2010]{
16W99, 
08B20 
16T10 
16T05  	
16T30  	
}

\keywords{
 Hopf algebras; Operated algebras; Planar binary trees
}

\maketitle
\vspace{-1.2cm}

\tableofcontents

\vspace{-1.3cm}

\allowdisplaybreaks


\section{Introduction}
\smallskip
The rooted tree is a significant object studied in algebra and combinatorics.
Many algebraic structures have been equipped on rooted trees. One of the most important
examples is the Connes-Kreimer Hopf algebra~\mcite{CK98},  which is employed to deal with a problem
of renormalization in Quantum Field Theory~\mcite{BF10, CF11, CK00, FGK, GPZ1, Kre98}.
Other Hopf algebras have also been constructed on rooted trees in different situations, such as Loday-Ronco~\mcite{LR98},
Grossman-Larson~\mcite{GL89}, Foissy-Holtkamp~\mcite{Foi02,Fois02,Hol03}.
Furthermore other algebraic structures, such as dendriform algebras~\mcite{Lod93}, pre-Lie algebras~\mcite{CL01},
operated algebras~\mcite{Guo09}, and Rota-Baxter algebras~\mcite{ZGG16}, have been established on rooted trees.
Most of these algebraic structures possess certain universal properties. For example, the Connes-Kreimer Hopf algebra of rooted trees
inherits its algebra structure from the initial object in the category of (commutative) algebras with a linear operator~\mcite{Foi02,Moe01}.

As a special case of rooted trees, (rooted) planar binary trees play an indispensable role in the study of combinatorics~\mcite{Sta97}, algebraic operads~\mcite{Cha07, LV03}, associahedrons~\mcite{LR04}, cluster algebras~\mcite{HLT11} and  Hopf algebras~\mcite{AS06, BF03, LR98}. In \mcite{LR98}, Loday and Ronco defined a Hopf algebra $H_{LR}$ (with unity) on planar binary trees, which is a free associative algebra on the trees of the form $|\vee T$,
that is the trees such that the tree born from the root on the left has only one leaf.
The $H_{LR}$ (without unity) is the free dendriform algebra on one generator~\mcite{LR98, LR01}. Later, Brouder and Frabetti~\mcite{BF03} showed that there exists a noncommutative Hopf algebra on planar binary trees which represents the renormalization group of quantum electrodynamics, and the coaction which describes the renormalization procedure. In the algebraic framework of Chapoton~\mcite{Cha04} for Bessel operad, a Hopf operad is constructed on the vector spaces spanned by forests of leaf-labeled binary rooted trees.
Aguiar and Sottile further studied the structure of the Loday-Ronco Hopf algebra by a new basis in~\mcite{AS06}, where the product, coproduct and antipode in terms of this basis were also given.

The concept of an algebra with (one or more) linear operators was introduced by Kurosh~\cite{Kur60}.
Later Guo~\cite{Guo09} constructed the free objects of such algebras in terms of various combinatorial objects,
such as Motzkin paths, rooted forests and bracketed words by the name of $\X$-operated algebras,
where $\X$ is a nonempty set used to index the operators. See also~\mcite{BCQ10,GG17, Gub}.
The Connes-Kreimer Hopf algebra of rooted trees can be viewed as an operated algebra, where the operator is the grafting operation $B^+$.
More generally, the decorated (planar) rooted trees with vertices decorated by a set $\X$,
together with a set of grafting operations $\{B^+_\alpha\mid \alpha\in \X\})$, is an $\X$-operated algebra~\mcite{KP, ZGG16}.
Indeed it is the free $\X$-operated algebra on the empty set or equivalently the initial object in the category of $\X$-operated algebras.

It is well-known that the noncommutative Connes-Kreimer Hopf algebra of planar rooted trees is isomorphic to the Loday-Ronco Hopf algebra of planar binary trees~\mcite{Fois02, Hol03}. Now the former can be treated in the framework of operated algebras~\mcite{ZGG16}. So there should be an analogy of operated alebras on top of planar binary trees, which is introduced and explored in the present paper by the name of $\vee$-algebras or more generally $\vee_\X$-algebras.
Let us emphasize that the binary grafting operation $\vee$ on planar binary trees has subtle difference with the aforementioned grafting operation $B^+$ on rooted trees---the $\vee$ is binary while $B^+$ is unary.
Thanks to these new concepts, the decorated planar binary trees $H_{LR}(\X)$ can viewed as a free cocycle $\vee_\X$-bialgebra and further a free cocycle $\vee_\X$-Hopf algebra on the empty set, involving an analogues of a Hochschild $1$-cocycle condition on planar rooted trees~\mcite{Fo3}.
In particular, the well-known Loday-Ronco Hopf algebra $H_{LR}$ is a free cocycle $\vee$-Hopf algebra.
This new free algebraic structure on planar binary trees validates again that most of algebraic structures on rooted trees have universal properties.

Our second source of inspiration and motivation is the admissible cut on rooted trees which was introduced by Connes and Kreimer~\mcite{CK98}. We adapt from this cut to expose the concept of admissible cut on decorated planar binary trees. Surprisingly, the admissible cuts on decorated planar binary trees,  make it possible to give a combinatorial description of the coproduct on the decorated Loday-Ronco Hopf algebras. We point out that our admissible cut is different from the one introduced by Connes and Kreimer~\mcite{CK98}, see Remark~\mref{remk:admi}.

{\bf Structure of the Paper.} In Section~\mref{sec:hadpbt}, we first recall some results concerning the Hopf algebraic structures on decorated planar binary trees. Motivated by the admissible cut on rooted trees, we introduce the concept of an admissible cut on  decorated planar binary trees. Having this concept in hand, we give a combinatorial description of the coproduct of the decorated Loday-Ronco Hopf algebra (Theorem~\mref{them:comb}).
We end this section by showing that $H_{LR}(\X)$ is a strictly graded coalgebra concerning the coalgebra structure (Theorem~\mref{them:sgc}). In Section~\mref{sec:fch}, viewing the Hopf algebra of decorated planar binary trees in the framework of operated algebras, we build  $\vee$-algebras and more generally $\vee_\X$-algebras (Definition~\mref{defn:veealg}), leading to the notations of (cocycle) $\vee_\X$-bialgebras and $\vee_\X$-Hopf algebras (Definitions~\mref{defn:veehopf0}, \mref{defn:veehopf}), involving a $\vee$-cocycle condition. With the help of these concepts, we first equip the decorated planar binary trees $H_{LR}(\X)$ with a free $\vee_\X$-algebraic structure (Theorem~\mref{thm:fva}). A family of coideals of a $\vee_\X$-bialgebra is also given (Proposition~\mref{pp:coideal}). We then prove respectively that $H_{LR}(\X)$ is the free cocycle $\vee_\X$-bialgebra and free cocycle $\vee_\X$-Hopf algebra on the empty set (Theorem~\mref{thm:conclude}). In particular when $\X$ is a singleton set, we establish respectively the free cocycle $\vee$-bialgebra and free cocycle $\vee$-Hopf algebra structures on the well-known Loday-Ronco Hopf algebra $H_{LR}$ (Corollary~\mref{cor:clr}).

{\bf Convention. } Throughout this paper, let $\bfk$ be a unitary commutative ring which will be the base ring of all  modules, algebras,  coalgebras and bialgebras, as well as linear maps. Algebras are  unitary algebras but not necessary commutative.
For any set $Y$, denote by $\bfk Y$ the free \bfk-module with basis $Y$.

\section{Hopf algebras of decorated planar binary trees}
\mlabel{sec:hadpbt}
In this section, we expose some results and notations concerning
Hopf algebraic structures on decorated planar binary trees, which will be used later.
See~\cite{Cha07, Fois02, Man08, Ron02} for more details.

\subsection{Hopf algebras of decorated planar binary trees}
A ${\bf planar\ tree}$ is an oriented graph draw on a plane, with  a preferred vertex called the $\mathbf{root}$. It is {\bf binary} when any vertex is trivalent (one root and two leaves)~\mcite{LR98}. The root is at the bottom of the tree. For each $n\geq 0$, the set of planar binary trees with $n$ interior
vertices will be denoted by $Y_n$. For instance,
\begin{align*}
Y_0&=\{|\},\ \
Y_1=\left\{ \YY{} \right\},\ \
Y_2=\left\{ \YY{\yy10}, \, \YY{\yy11} \right\},\\
Y_3&=\left\{ \YY{\yy10\yy20}, \, \YY{\yy11\yy23}, \, \YY{\yy20\yy23}, \, \YY{\yy10\yy21}, \, \YY{\yy11\yy22} \right\}.
\end{align*}
Here $|$ stands for the unique tree with one leaf. The number of the set $Y_n$ is given by the Catalan number $\frac{(2n)!}{n!(n+1)!}$~\cite{LR98}.

Let $\X$ be a nonempty set throughout the remainder of the paper. For each $n\geq 0$, let $Y_{n}(\X)$ denote the set of planar binary trees in $Y_n$
with interior vertices decorated by elements of $\X$.
Denote by $$Y_\infty(\X):= \bigsqcup_{n\geq 0}Y_n(\X)\,\text{ and }\,  H_{LR}(\X):= \bfk Y_{\infty}(\X) = \bigoplus_{n\geq 0} \bfk Y_{n}(\X).$$

A planar binary tree $T$ in $Y_n(\X)$ is called an {\bf $n$-decorated planar binary tree} or {\bf $n$-tree} for simplicity.
The $\mathbf{depth}$ $\mathrm{dep}(T)$ of a decorated planar binary tree $T$ is the maximal length of linear chains from the root to the leaves of the tree.
For example,
$$\dep(|)=0\,\text{ and }\, \dep({\YY{\node[scale=0.8] at (0.75,-0.55) {{$\alpha$}};}}) = 1.$$

Let $T\in Y_{m}(\X)$ and  $T'\in Y_{n}(\X)$ be two decorated planar binary trees and $\alpha$ an element in $\X$. The
grafting $\vee_\alpha$ of $T$ and $T'$ on $\alpha$ is the $(n + m + 1)$-decorated planar binary tree $T\vee_\alpha T'\in Y_{m+n+1}(\X)$,
obtained by joining the roots of $T$ and $T'$ and create a new root, which is decorated by $\alpha$.
For any decorated planar binary tree $T\in Y_{n}(\X)$ with $n\geq 1$, there exist unique elements
$T^l\in Y_{k}(\X)$, $T^r\in Y_{n-k-1}(\X)$ and $\alpha\in \X$ such that
\begin{align*}
T=T^l\vee_{\alpha}T^r,
\end{align*}
where $T^l$ and $T^r$ are the left-hand side of $T$ and the right-hand side of $T$, respectively. For instance
\begin{align*}
\YY{\node[scale=0.8] at (0.75,-0.55) {${\beta}$};}\vee_{\alpha}|=\YY{\yy10
\node[scale=0.8] at (0.75,-0.55) {${\alpha}$};
\node[scale=0.8] at (0.05,-0.4) {${\beta}$};
}, \quad
\YY{\node[scale=0.8] at (0.75,-0.55) {${\beta}$};}\vee_{\alpha}\YY{\node[scale=0.8] at (0.75,-0.55) {${\gamma}$};}=
\YY{\yy20\yy23
\node[scale=0.8] at (-0.1,-0.25) {${\beta}$};
\node[scale=0.8] at (1.10,-0.25) {${\gamma}$};
\node[scale=0.8] at (0.75,-0.55) {${\alpha}$};
}.
\end{align*}

A multiplication $\ast$ on $H_{LR}(\X)$ with unit $|$ is given recursively on the sum of depth as~\cite[Sec.~4.3]{Fois02}
\begin{align}
|\ast T :=T\ast |:= T\ \text{\ and\ } T\ast T':=T^l\vee_\alpha(T^r\ast T')+(T\ast T^{'l})\vee_\beta T^{'r},
\mlabel{eq:astttt}
\end{align}
where $T=T^l\vee_\alpha T^r$ and $T'=T^{'l}\vee_\beta T^{'r}$ are in $Y_\infty(\X)$ with $\alpha,\beta\in \X$.
Let us agree to fix the notation $\ast$ to denote the multiplication given in Eq.~(\mref{eq:astttt}) hereafter.

\begin{exam} We have
\begin{align*}
&\YY{\node[scale=0.8] at (0.75,-0.55) {${\alpha}$};}\ast \YY{\node[scale=0.8] at (0.75,-0.55) {${\beta}$};}=\YY{\yy11
\node[scale=0.8] at (0.75,-0.55) {${\alpha}$};
\node[scale=0.8] at (0.95,-0.30) {${\beta}$};
}+
\YY{\yy10
\node[scale=0.8] at (0.75,-0.55) {${\beta}$};
\node[scale=0.8] at (0.04,-0.3) {${\alpha}$};
}, \quad
\YY{\yy10
\node[scale=0.8] at (0.75,-0.55) {${\alpha}$};
\node[scale=0.8] at (0.03,-0.35) {${\beta}$};
}\ast
\YY{\node[scale=0.8] at (0.75,-0.55) {${\gamma}$};}=
\YY{\yy20\yy23
\node[scale=0.8] at (-0.1,-0.25) {${\beta}$};
\node[scale=0.8] at (1.10,-0.25) {${\gamma}$};
\node[scale=0.8] at (0.75,-0.55) {${\alpha}$};
}+
\YY{\yy10\yy20
\node[scale=0.8] at (-0.1,-0.2) {${\beta}$};
\node[scale=0.8] at (0.15,-0.39) {${\alpha}$};
\node[scale=0.8] at (0.75,-0.55) {${\gamma}$};
}, \\
&\YY{\node[scale=0.8] at (0.75,-0.55) {${\gamma}$};}\ast
\YY{\yy10
\node[scale=0.8] at (0.75,-0.55) {${\alpha}$};
\node[scale=0.8] at (0.03,-0.35) {${\beta}$};
}=
\YY{\yy11\yy22
\node[scale=0.8] at (0.75,-0.65) {${\gamma}$};
\node[scale=0.8] at (0.88,-0.35) {${\alpha}$};
\node[scale=0.7] at (0.45,-0.21) {${\beta}$};
}+
\YY{\yy10\yy21
\node[scale=0.8] at (0.75,-0.55) {${\alpha}$};
\node[scale=0.8] at (0.05,-0.4) {${\gamma}$};
\node[scale=0.7] at (0.55,-0.2) {${\beta}$};
}+
\YY{\yy10\yy20
\node[scale=0.8] at (-0.1,-0.2) {${\gamma}$};
\node[scale=0.8] at (0.15,-0.39) {${\beta}$};
\node[scale=0.8] at (0.75,-0.55) {${\alpha}$};
}.
\end{align*}
\end{exam}

In the undecorated case, the description of the coproduct in the Loday-Ronco Hopf algebra $H_{LR}$ was first introduced in~\cite[Proposition 3.3]{LR98}.
In the decorated case, Foissy~\cite[Sec.~4.3]{Fois02} equipped the $\mathbf{k}$-algebra $H_{LR}(\X)$  with a coproduct $\Delta_{LR(\X)}$ described recursively on $\dep(T)$
for basis elements $T\in Y_\infty(\X)$ as
\begin{align}
\Delta_{LR(\X)}(T):=|\otimes|  \,\text{ if } T=|;
\mlabel{eq:dlr11}
\end{align}
and for $T = T^l\vee_{\alpha}T^r$,
\begin{align}
\Delta_{LR(\X)}(T) :=\Delta_{LR(\X)}(T^l\vee_{\alpha}T^r):=(T^l\vee_{\alpha}T^r)\otimes |+(\ast, \vee_\alpha)\bigl(\Delta_{LR(\X)}(T^l)\otimes \Delta_{LR(\X)}(T^r)\bigr),
\mlabel{eq:dlr22}
\end{align}
where $(\ast, \vee_\alpha):=(\ast \otimes \vee_\alpha)\circ {\tau_{23}}$ and $\tau_{23}$ is the permutation of the second and third tensor factors.

\begin{exam}
We have
\begin{align*}
\Delta_{LR(\X)}(\YY{\node[scale=0.8] at (0.75,-0.55) {${\alpha}$};})&=\YY{\node[scale=0.8] at (0.75,-0.55) {${\alpha}$};}\otimes |+|\otimes \YY{\node[scale=0.8] at (0.75,-0.55) {${\alpha}$};},\\
\Delta_{LR(\X)}(\YY{\yy11
\node[scale=0.8] at (0.75,-0.55) {${\alpha}$};
\node[scale=0.8] at (0.95,-0.35) {${\beta}$};
})&=
\YY{\yy11
\node[scale=0.8] at (0.75,-0.55) {${\alpha}$};
\node[scale=0.8] at (0.95,-0.35) {${\beta}$};
}\otimes |+|\otimes
\YY{\yy11
\node[scale=0.8] at (0.75,-0.55) {${\alpha}$};
\node[scale=0.8] at (0.95,-0.35) {${\beta}$};
}+
\YY{\node[scale=0.8] at (0.75,-0.55) {${\beta}$};}\otimes
\YY{\node[scale=0.8] at (0.75,-0.55) {${\alpha}$};}
,\\
\Delta_{LR(\X)}(\YY{\yy10
\node[scale=0.8] at (0.75,-0.55) {${\alpha}$};
\node[scale=0.8] at (0.03,-0.35) {${\beta}$};
})&=
\YY{\yy10
\node[scale=0.8] at (0.75,-0.55) {${\alpha}$};
\node[scale=0.8] at (0.03,-0.35) {${\beta}$};
}\otimes |
+|\otimes
\YY{\yy10
\node[scale=0.8] at (0.75,-0.55) {${\alpha}$};
\node[scale=0.8] at (0.03,-0.35) {${\beta}$};
}+
\YY{\node[scale=0.8] at (0.25,-0.55) {${\beta}$};}\otimes \YY{\node[scale=0.8] at (0.75,-0.55) {${\alpha}$};},\\
\Delta_{LR(\X)}(\YY{\yy10\yy21
\node[scale=0.8] at (0.75,-0.55) {${\alpha}$};
\node[scale=0.8] at (0.05,-0.4) {${\beta}$};
\node[scale=0.7] at (0.53,-0.2) {${\gamma}$};
})&=
\YY{\yy10\yy21
\node[scale=0.8] at (0.75,-0.55) {${\alpha}$};
\node[scale=0.8] at (0.05,-0.4) {${\beta}$};
\node[scale=0.7] at (0.53,-0.2) {${\gamma}$};
}\otimes |+|\otimes
\YY{\yy10\yy21
\node[scale=0.8] at (0.75,-0.55) {${\alpha}$};
\node[scale=0.8] at (0.05,-0.4) {${\beta}$};
\node[scale=0.7] at (0.53,-0.2) {${\gamma}$};
}+
\YY{\yy11
\node[scale=0.8] at (0.25,-0.55) {${\beta}$};
\node[scale=0.8] at (0.95,-0.35) {${\gamma}$};
}\otimes\YY{\node[scale=0.8] at (0.75,-0.55) {${\alpha}$};}
+\YY{\node[scale=0.8] at (0.75,-0.55) {${\gamma}$};}\otimes
\YY{\yy10
\node[scale=0.8] at (0.75,-0.55) {${\alpha}$};
\node[scale=0.8] at (0.03,-0.35) {${\beta}$};
},\\
\Delta_{LR(\X)}(\YY{\yy20\yy23
\node[scale=0.8] at (-0.1,-0.25) {${\beta}$};
\node[scale=0.8] at (1.10,-0.25) {${\gamma}$};
\node[scale=0.8] at (0.75,-0.55) {${\alpha}$};
})&=
\YY{\yy20\yy23
\node[scale=0.8] at (-0.1,-0.25) {${\beta}$};
\node[scale=0.8] at (1.10,-0.25) {${\gamma}$};
\node[scale=0.8] at (0.75,-0.55) {${\alpha}$};
}\otimes |+|\otimes
\YY{\yy20\yy23
\node[scale=0.8] at (-0.1,-0.25) {${\beta}$};
\node[scale=0.8] at (1.10,-0.25) {${\gamma}$};
\node[scale=0.8] at (0.75,-0.55) {${\alpha}$};
}+
(\YY{\yy11
\node[scale=0.8] at (0.25,-0.55) {${\beta}$};
\node[scale=0.8] at (0.95,-0.35) {${\gamma}$};
}+
\YY{\yy10
\node[scale=0.8] at (0.75,-0.55) {${\gamma}$};
\node[scale=0.8] at (0.03,-0.35) {${\beta}$};
})\otimes \YY{\node[scale=0.8] at (0.75,-0.55) {${\alpha}$};}
+\YY{\node[scale=0.8] at (0.25,-0.55) {${\beta}$};}\otimes
\YY{\yy11
\node[scale=0.8] at (0.75,-0.55) {${\alpha}$};
\node[scale=0.8] at (0.95,-0.35) {${\gamma}$};
}+\YY{\node[scale=0.8] at (0.75,-0.55) {${\gamma}$};}\otimes\YY{\yy10
\node[scale=0.8] at (0.75,-0.55) {${\alpha}$};
\node[scale=0.8] at (0.03,-0.35) {${\beta}$};
}.
\end{align*}
\mlabel{ex:del}
\end{exam}

Foissy~\mcite{Fois02} also defined linear maps
\begin{align*}
\varepsilon_{LR(\X)}: H_{LR}(\X)\rightarrow \mathbf{k}, \quad | \mapsto 1_\bfk \,\text{ and }\, T \mapsto 0 \,\text{ for }\, |\neq T\in Y_\infty(\X)
\end{align*}
and
\begin{align*}
|: \bfk \to H_{LR}(\X), \quad 1_\bfk \mapsto |.
\end{align*}

Recall~\cite{Man08} that a bialgebra $(H, \ast_H, 1_H, \Delta, \varepsilon)$ is called $\mathbf{graded}$ if there are $\mathbf{k}$-submodules $H^{(n)}, n\geq 0$, of $H$ such that
\begin{enumerate}
\item $H=\bigoplus_{n=0}^{\infty}H^{(n)}$;
\vskip 0.1in
\item $H^{(p)}H^{(q)}\subseteq H^{(p+q)}, \, p, q \geq 0$; and
\vskip 0.1in
\item $\Delta(H^{(n)})\subseteq \bigoplus_{p+q=n}H^{(p)}\otimes H^{(q)}, n\geq 0$.
\end{enumerate}
Elements of $H^{(n)}$ are called to have degree $n$. $H$ is called $\mathbf{connected}$ if $H^{(0)}=\mathbf{k}$
and $\ker \varepsilon=\bigoplus_{n\geq 1}H^{(n)}$.
It is well-known that a connected graded bialgebra is a Hopf algebra~\mcite{DM}.

\begin{lemma}\cite[Sec.~4.3]{Fois02}~\cite[Sec.~6.3.5]{Man08}
The quintuple $(H_{LR}(\X), \ast, |, \Delta_{LR(\X)}, \varepsilon_{LR(\X)})$ is a connected graded bialgebra with grading $H_{LR}(\X)= \oplus_{n\geq 0} \bfk Y_{n}(\X)$ and hence a Hopf algebra.
\mlabel{lem:hopf1}
\end{lemma}

If $\X$ is a singleton set, then $Y_\infty(\X)$ is precisely the planar binary trees (without decorations)
and one gets the Loday-Ronco Hopf algebra on planar binary trees~\cite[Thm.~3.1]{LR98}.
\subsection{A combinatorial description of $\Delta_{LR(\X)}$}
Next we give a combinatorial description of the coproduct $\Delta_{LR(\X)}$ by the admissible cut which was introduced by Connes and Kreimer~\mcite{CK98} on rooted trees and further studied by Foissy~\mcite{Foi10} on decorated rooted trees. This notion of cut of rooted trees can be adapted
to decorated planar binary trees as follows.

Let $T\in Y_{\infty}(\X)$ be a decorated planar binary tree. The edges of $T$ are oriented upwards, from root to leaves.
A {\bf (non-total) cut} $c$ is a choice of edges connecting internal
vertices of $T$. Note that an edge connecting a leaf and an internal vertex is not in a cut.
In particular, the {\bf empty cut} is a cut with the choice of no edges. The cut $c$ is called {\bf admissible} if any oriented path from a vertex of the tree to the root meets at most one cut edge.
For an admissible cut $c$, cutting each edge in $c$ into two edges, $T$ is sent to a pair $(P^c(T), R^c(T))$, such that
$R^c(T)$ is the connected component containing the root of $T$ and $P^c(T)$ is the product of the other connected components with respect to
the multiplication $\ast$ given in Eq.~(\mref{eq:astttt}), from left to right. The {\bf total cut} is also added, which is by convention an admissible cut such that
$$R^c(T)=| \text{ and } P^c(T)=T.$$
The set of admissible cuts of $T$ is denoted by $\adm_*(T)$. Let us note that the empty cut is admissible.
Denote by
$$ \adm(T) :=  \adm_*(T) \setminus \{\mbox{empty cut, total cut}\}.$$

\begin{exam}
\begin{enumerate}
\item \mlabel{it:a}
Consider the decorated planar binary tree $T=\YY[line width=0.16ex]{\yy10\yy21
\node[scale=0.8] (x) at (0.75,-0.55) {$\alpha$};
\node[scale=0.8] (y) at (0.1,-0.4) {$\beta$};
\node[scale=0.8] (z) at (0.55,-0.15) {$\gamma$};
}$ with $\alpha, \beta, \gamma \in \X$. It has $2^2$ non-total cuts and one total cut.
$$\begin{array}{|c|c|c|c|c|c|}
\hline \mbox{cut }c&
\mbox{empty}
&\YY[line width=0.16ex]{\yy10\yy21
\node[scale=0.8] (x) at (0.75,-0.55) {$\alpha$};
\node[scale=0.8] (y) at (0.1,-0.4) {$\beta$};
\node[scale=0.8] (z) at (0.55,-0.15) {$\gamma$};
\draw (0.18,-0.375)--+(0.35,0);
}
&\YY[line width=0.16ex]{\yy10\yy21
\node[scale=0.8] (x) at (0.75,-0.55) {$\alpha$};
\node[scale=0.8] (y) at (0.1,-0.4) {$\beta$};
\node[scale=0.8] (z) at (0.55,-0.15) {$\gamma$};
\draw (0.25,-0.1875)--+(0.23,0);
}
&\YY[line width=0.16ex]{\yy10\yy21
\node[scale=0.8] (x) at (0.75,-0.55) {$\alpha$};
\node[scale=0.8] (y) at (0.1,-0.4) {$\beta$};
\node[scale=0.8] (z) at (0.55,-0.15) {$\gamma$};
\draw (0.18,-0.375)--+(0.35,0);
\draw (0.25,-0.1875)--+(0.23,0);
}
&\mbox{total}\\
\hline \mbox{Admissible?}&\mbox{yes}&\mbox{yes}&\mbox{yes}&\mbox{no}&\mbox{yes}\\
\hline R^c(T)
&\YY[line width=0.16ex]{\yy10\yy21
\node[scale=0.8] (x) at (0.75,-0.55) {$\alpha$};
\node[scale=0.8] (y) at (0.1,-0.4) {$\beta$};
\node[scale=0.8] (z) at (0.55,-0.15) {$\gamma$};
}
&\YY{\node[scale=0.8] at (0.75,-0.55) {${\alpha}$};}
&\YY{\yy10
\node[scale=0.8] at (0.75,-0.55) {${\alpha}$};
\node[scale=0.8] at (0.04,-0.3) {${\beta}$};
}
&\times&|\\
\hline P^c(T)&|
&\YY{\yy11
\node[scale=0.8] at (0.25,-0.55) {${\beta}$};
\node[scale=0.8] at (0.95,-0.30) {${\gamma}$};
}
&\YY{\node[scale=0.8] at (0.75,-0.55) {${\gamma}$};}
&\times&
\YY[line width=0.16ex]{\yy10\yy21
\node[scale=0.8] (x) at (0.75,-0.55) {$\alpha$};
\node[scale=0.8] (y) at (0.1,-0.4) {$\beta$};
\node[scale=0.8] (z) at (0.55,-0.15) {$\gamma$};
}\\
\hline \end{array}$$

\item \mlabel{it:b}
Consider the decorated planar binary tree $T=\YY{\yy20\yy23
\node[scale=0.8] at (-0.1,-0.25) {${\beta}$};
\node[scale=0.8] at (1.10,-0.25) {${\gamma}$};
\node[scale=0.8] at (0.75,-0.55) {${\alpha}$};
}$ with $\alpha, \beta, \gamma \in \X$. It has $2^2$ non-total cuts and one total cut.
$$\begin{array}{|c|c|c|c|c|c|}
\hline \mbox{cut }c&
\mbox{empty}
&\YY[line width=0.16ex]{\yy20\yy23
\node[scale=0.8] at (-0.1,-0.25) {${\beta}$};
\node[scale=0.8] at (1.10,-0.25) {${\gamma}$};
\node[scale=0.8] at (0.75,-0.55) {${\alpha}$};
\draw (0.1,-0.275)--+(0.35,0);
}
&\YY[line width=0.16ex]{\yy20\yy23
\node[scale=0.8] at (-0.1,-0.25) {${\beta}$};
\node[scale=0.8] at (1.10,-0.25) {${\gamma}$};
\node[scale=0.8] at (0.75,-0.55) {${\alpha}$};
\draw (0.5,-0.275)--+(0.515,0);
}
&\YY[line width=0.16ex]{\yy20\yy23
\node[scale=0.8] at (-0.1,-0.25) {${\beta}$};
\node[scale=0.8] at (1.10,-0.25) {${\gamma}$};
\node[scale=0.8] at (0.75,-0.55) {${\alpha}$};
\draw (0.1,-0.275)--+(0.35,0);
\draw (0.5,-0.275)--+(0.515,0);
}
&\mbox{total}\\
\hline \mbox{Admissible?}&\mbox{yes}&\mbox{yes}&\mbox{yes}&\mbox{yes}&\mbox{yes}\\
\hline R^c(T)
&\YY[line width=0.16ex]{\yy20\yy23
\node[scale=0.8] at (-0.1,-0.25) {${\beta}$};
\node[scale=0.8] at (1.10,-0.25) {${\gamma}$};
\node[scale=0.8] at (0.75,-0.55) {${\alpha}$};
}
&\YY{\yy11
\node[scale=0.8] at (0.75,-0.55) {${\alpha}$};
\node[scale=0.8] at (0.95,-0.30) {${\gamma}$};
}
&\YY{\yy10
\node[scale=0.8] at (0.75,-0.55) {${\alpha}$};
\node[scale=0.8] at (0.04,-0.3) {${\beta}$};
}
&\YY{\node[scale=0.8] at (0.75,-0.55) {${\alpha}$};}
&|\\
\hline P^c(T)&|
&\YY{\node[scale=0.8] at (0.25,-0.55) {${\beta}$};}
&\YY{\node[scale=0.8] at (0.75,-0.55) {${\gamma}$};}
&\YY{\node[scale=0.8] at (0.25,-0.55) {${\beta}$};}\ast\YY{\node[scale=0.8] at (0.75,-0.55) {${\gamma}$};}
&
\YY[line width=0.16ex]{\yy20\yy23
\node[scale=0.8] at (-0.1,-0.25) {${\beta}$};
\node[scale=0.8] at (1.10,-0.25) {${\gamma}$};
\node[scale=0.8] at (0.75,-0.55) {${\alpha}$};
}\\
\hline \end{array}$$
\end{enumerate}
\mlabel{exam:comb2}
\end{exam}

\begin{remark}
It should be pointed out that our admissible cut is different from the one,
which is introduced by Connes-Kreimer on undecorated planar rooted
trees~\mcite{CK98} and further studied by Foissy on decorated planar rooted trees~\mcite{Foi10}.
For example, under the framework of~\mcite{Foi10}, Foissy gave
$$R^c(~\tquatredeuxa~)=\tddeux{$\alpha$}{$\gamma$}\,  \text{ and }\,P^c(~\tquatredeuxa~)=\tddeux{$\beta$}{$\delta$}. $$
The undecorated case can also be found in ~\cite[Figure~5]{CK98}. Note that the cutting edge is deleted.
However our admissible cut $c$ cuts each cutting edge into two edges.
\mlabel{remk:admi}
\end{remark}

Now we are ready to give a combinatorial description of the coproduct $\Delta_{LR(\X)}$.

\begin{theorem}
Let $T \in  Y_{\infty}(\X)\setminus \{|\}$. Then
\begin{align}
\Delta_{LR(\X)}(T)=\sum_{c\in  \adm_*(T)} P^c(T)\otimes R^c(T)=T\otimes |+|\otimes T+\sum_{c \in  \adm(T)} P^c(T) \otimes R^c(T).
\mlabel{eq:comb}
\end{align}
\mlabel{them:comb}
\end{theorem}

\begin{proof}
We prove Eq.~(\mref{eq:comb}) by induction on the depth $\dep(T)\geq 1$. For the initial step of $\dep(T)=1$, we have $T=\YY{\node[scale=0.8] at (0.75,-0.55) {${\alpha}$};}$ for some $\alpha\in \X$. Since $\YY{\node[scale=0.8] at (0.75,-0.55) {${\alpha}$};}$ has only one internal vertex and each edge in a cut can't
connect a leaf and an internal vertex, $T$ has only two cuts---the total cut and the empty cut. Thus $\adm_*(T)$ consists of the total cut and
the empty cut, and so
\begin{align*}
\Delta_{LR(\X)}(\YY{\node[scale=0.8] at (0.75,-0.55) {${\alpha}$};})=\YY{\node[scale=0.8] at (0.75,-0.55) {${\alpha}$};}\otimes |+|\otimes \YY{\node[scale=0.8] at (0.75,-0.55) {${\alpha}$};}.
\end{align*}

For the induction step of $\dep(T)\geq 2$, we may write $T = T^l\vee_{\alpha}T^r$ for some $T^l, T^r\in Y_{\infty}(\X)$ and $\alpha\in \X$.
We have two cases to consider.

\noindent{\bf Case 1.} $\dep(T^l)=0$ and $\dep(T^r)\geq 1$, or $\dep(T^l)\geq 1$ and $\dep(T^r)=0$. Without loss of generality, we consider
$\dep(T^l)=0$ and $\dep(T^r)\geq 1$. Then $T^l=|$ and $T = |\vee_{\alpha}T^r$. By Eq.~(\mref{eq:dlr22}),
\allowdisplaybreaks{
\begin{align*}
\Delta_{LR(\X)}(T)=& \ \Delta_{LR(\X)}(|\vee_{\alpha}T^r)=(|\vee_{\alpha}T^r)\otimes |+(\ast, \vee_\alpha)\bigl(\Delta_{LR(\X)}(|)\otimes \Delta_{LR(\X)}(T^r)\bigr)\\
=&\ T\otimes |+(\ast, \vee_\alpha)\bigl(|\ot |\otimes \Delta_{LR(\X)}(T^r)\bigr)\quad\quad  (\text{by\  Eq.~(\mref{eq:dlr11})} )\\
=&\ T\otimes |+(\ast, \vee_\alpha)\bigg(|\ot |\otimes \Big(T^r\otimes |+|\otimes T^r+\sum_{c \in  \adm(T^r)} P^c(T^r) \otimes R^c(T^r)\Big)\bigg)\\
&\hspace{7cm} (\text{by the induction hypothesis})\\
=&\ T\otimes |+(\ast, \vee_\alpha)\bigg(|\ot |\otimes T^r\otimes |+|\ot |\ot|\otimes T^r+\sum_{c \in  \adm(T^r)} |\ot |\ot P^c(T^r) \otimes R^c(T^r)\bigg)\\
=&\ T\otimes |+(|\ast T^r)\otimes (|\vee_{\alpha}|)+(|\ast |)\ot(|\vee_{\alpha} T^r)+\sum_{c \in  \adm(T^r)} (|\ast  P^c(T^r)) \otimes (|\vee_{\alpha} R^c(T^r))\\
=&\ T\otimes |+ |\ot(|\vee_{\alpha} T^r)+T^r\otimes (|\vee_{\alpha}|)+\sum_{c \in  \adm(T^r)}   P^c(T^r) \otimes (|\vee_{\alpha} R^c(T^r))\\
=&\ T\otimes |+ |\ot T+\sum_{c \in  \adm(T)}   P^c(T) \otimes  R^c(T^r).
\end{align*}
}
\noindent{\bf Case 2.} $\dep(T^l)\geq 1$ and $\dep(T^r)\geq 1$. Then $T = T^l\vee_{\alpha}T^r$ with $T^l\neq |$ and $T^r\neq |$.
It follows from Eq.~(\mref{eq:dlr22}) that
\allowdisplaybreaks{
\begin{align}\nonumber
\Delta_{LR(\X)}(T)=&\ \Delta_{LR(\X)}(T^l\vee_{\alpha}T^r)= (T^l\vee_{\alpha}T^r)\otimes |+(\ast, \vee_\alpha)\bigl(\Delta_{LR(\X)}(T^l)\otimes \Delta_{LR(\X)}(T^r)\bigr)\\\nonumber
=&\ T\otimes |+(\ast, \vee_\alpha)\biggl(\Big(T^l\otimes |+|\otimes T^l+\sum_{c \in  \adm(T^l)} P^c(T^l) \otimes R^c(T^l)\Big)\otimes \\\nonumber
&\ \Big(T^r\otimes |+|\otimes T^r+\sum_{c' \in  \adm(T^r)} P^{c'}(T^r) \otimes R^{c'}(T^r)\Big)\biggr)\hspace{1cm} (\text{by the induction hypothesis})\\ \nonumber
=&\ T\otimes |+(\ast, \vee_\alpha)\biggl(T^l\ot | \ot T^r\ot | +T^l\otimes |\ot |\otimes T^r+T^l\ot |\ot \sum_{c' \in  \adm(T^r)} P^{c'}(T^r) \otimes R^{c'}(T^r)\\\nonumber
&+|\otimes T^l \ot T^r\otimes |+|\otimes T^l \ot |\otimes T^r+|\otimes T^l \ot \sum_{c' \in  \adm(T^r)} P^{c'}(T^r) \otimes R^{c'}(T^r)\\\nonumber
&+\sum_{c \in  \adm(T^l)} P^c(T^l) \otimes R^c(T^l)\ot T^r\otimes |
+\sum_{c \in  \adm(T^l)} P^c(T^l)\ot R^c(T^l) \ot |\otimes T^r\\\nonumber
&+\sum_{c \in  \adm(T^l)} \sum_{c' \in  \adm(T^r)}P^{c}(T^l)\ot R^{c}(T^l)\ot  P^{c'}(T^r) \otimes R^{c'}(T^r)\biggr)\\\nonumber
=&\ T\otimes |+(T^l\ast T^r) \ot (| \vee_\alpha |) +T^l\otimes (|  \vee_\alpha T^r)+\sum_{c' \in  \adm(T^r)} (T^l\ast P^{c'}(T^r)) \otimes( |\vee_\alpha  R^{c'}(T^r))\\\nonumber
&+T^r\otimes (T^l\vee_{\alpha} |)+|\ot (T^l\vee_\alpha T^r)+\sum_{c' \in  \adm(T^r)} P^{c'}(T^r) \otimes (T^l\vee_\alpha R^{c'}(T^r))\\\nonumber
&+\sum_{c \in  \adm(T^l)}( P^c(T^l) \ast T^r)\otimes (R^c(T^l)\vee_\alpha |)
+\sum_{c \in  \adm(T^l)} P^c(T^l) \ot ( R^c(T^l)\vee_\alpha T^r) \\\nonumber
&+\sum_{c \in  \adm(T^l)} \sum_{c' \in  \adm(T^r)}(P^c(T^l)\ast P^{c'}(T^r)) \otimes (R^c(T^l)\vee_\alpha R^{c'}(T^r))\\ \mlabel{eq:comb2}
=&\ T\otimes |+|\ot T+(T^l\ast T^r) \ot (| \vee_\alpha |) +T^l\otimes (|  \vee_\alpha T^r)+T^r\otimes (T^l\vee_{\alpha} |)\\\nonumber
&+\sum_{c' \in  \adm(T^r)} (T^l\ast P^{c'}(T^r)) \otimes( |\vee_\alpha  R^{c'}(T^r))+\sum_{c' \in  \adm(T^r)} P^{c'}(T^r) \otimes (T^l\vee_\alpha R^{c'}(T^r))\\\nonumber
&+\sum_{c \in  \adm(T^l)}( P^c(T^l) \ast T^r)\otimes (R^c(T^l)\vee_\alpha |)
+\sum_{c \in  \adm(T^l)} P^c(T^l) \ot ( R^c(T^l)\vee_\alpha T^r) \\\nonumber
&+\sum_{c \in  \adm(T^l)} \sum_{c' \in  \adm(T^r)}(P^c(T^l)\ast P^{c'}(T^r)) \otimes (R^c(T^l)\vee_\alpha R^{c'}(T^r)).
\end{align}
}
We may draw the decorated planar binary tree $T$ graphically as
$$T=\YY{
\node[scale=0.8] at (0.75,-0.55) {${\alpha}$};
\node[above=-4pt] at (0,0) {$\renewcommand\arraystretch{1}\begin{array}{c}T_l\\ \smash\vdots\end{array}$};
\node[above=-4pt] at (1,0) {$\renewcommand\arraystretch{1}\begin{array}{c}T_r\\ \smash\vdots\end{array}$};
}.$$
Then all kinds of admissible cuts in $\adm(T)$ can be illustrated graphically as:
\begin{align*}
 \adm(T)= \left\{
\YY[baseline=0ex]{
\node[scale=0.8] at (0.75,-0.55) {${\alpha}$};
\node[above=-4pt] at (0,0) {$\renewcommand\arraystretch{1}\begin{array}{c}T_l\\ \smash\vdots\end{array}$};
\node[above=-4pt] at (1,0) {$\renewcommand\arraystretch{1}\begin{array}{c}T_r\\ \smash\vdots\end{array}$};
\draw (0.1,-0.2875)--+(0.3,0);
\draw (0.5,-0.2875)--+(0.52,0);
},
\YY[baseline=0ex] {
\node[scale=0.8] at (0.75,-0.55) {${\alpha}$};
\node[above=-4pt] at (0,0) {$\renewcommand\arraystretch{1}\begin{array}{c}T_l\\ \smash\vdots\end{array}$};
\node[above=-4pt] at (1,0) {$\renewcommand\arraystretch{1}\begin{array}{c}T_r\\ \smash\vdots\end{array}$};
\draw (0.1,-0.2875)--+(0.38,0);
},
\YY[baseline=0ex] {
\node[scale=0.8] at (0.75,-0.55) {${\alpha}$};
\node[above=-4pt] at (0,0) {$\renewcommand\arraystretch{1}\begin{array}{c}T_l\\ \smash\vdots\end{array}$};
\node[above=-4pt] at (1,0) {$\renewcommand\arraystretch{1}\begin{array}{c}T_r\\ \smash\vdots\end{array}$};
\draw (0.438,-0.2875)--+(0.4392,0);
},
\YY[baseline=0ex] {
\node[scale=0.8] at (0.75,-0.55) {${\alpha}$};
\node[above=-4pt] at (0,0) {$\renewcommand\arraystretch{1}\begin{array}{c}T_l\\ \smash\vdots\end{array}$};
\node[above=-4pt] at (1,0) {$\renewcommand\arraystretch{1}\begin{array}{c}T_r\\ \smash\vdots\end{array}$};
\draw (0.1,-0.2875)--+(0.38,0);
\draw (0.654,0.3975)--+(0.655,0);
},
\YY[baseline=0ex] {
\node[scale=0.8] at (0.75,-0.55) {${\alpha}$};
\node[above=-4pt] at (0,0) {$\renewcommand\arraystretch{1}\begin{array}{c}T_l\\ \smash\vdots\end{array}$};
\node[above=-4pt] at (1,0) {$\renewcommand\arraystretch{1}\begin{array}{c}T_r\\ \smash\vdots\end{array}$};
\draw (0.654,0.3975)--+(0.655,0);
},
\YY[baseline=0ex] {
\node[scale=0.8] at (0.75,-0.55) {${\alpha}$};
\node[above=-4pt] at (0,0) {$\renewcommand\arraystretch{1}\begin{array}{c}T_l\\ \smash\vdots\end{array}$};
\node[above=-4pt] at (1,0) {$\renewcommand\arraystretch{1}\begin{array}{c}T_r\\ \smash\vdots\end{array}$};
\draw (-0.3,0.3975)--+(0.58,0);
\draw (0.438,-0.2875)--+(0.4392,0);
},
\YY[baseline=0ex] {
\node[scale=0.8] at (0.75,-0.55) {${\alpha}$};
\node[above=-4pt] at (0,0) {$\renewcommand\arraystretch{1}\begin{array}{c}T_l\\ \smash\vdots\end{array}$};
\node[above=-4pt] at (1,0) {$\renewcommand\arraystretch{1}\begin{array}{c}T_r\\ \smash\vdots\end{array}$};
\draw (-0.3,0.3975)--+(0.58,0);
},
\YY[baseline=0ex] {
\node[scale=0.8] at (0.75,-0.55) {${\alpha}$};
\node[above=-4pt] at (0,0) {$\renewcommand\arraystretch{1}\begin{array}{c}T_l\\ \smash\vdots\end{array}$};
\node[above=-4pt] at (1,0) {$\renewcommand\arraystretch{1}\begin{array}{c}T_r\\ \smash\vdots\end{array}$};
\draw (0.654,0.3975)--+(0.655,0);
\draw (-0.3,0.3975)--+(0.58,0);}
\right\}.
\end{align*}
Note that the last eight terms in Eq.~(\mref{eq:comb2}) are precisely corresponding to the eight kinds of admissible cuts in $ \adm(T)$.
Thus
\begin{align*}
\Delta_{LR(\X)}(T)=\Delta_{LR(\X)}(T^l\vee_{\alpha}T^r)=T\otimes |+|\otimes T+\sum_{c \in  \adm(T)} P^c(T) \otimes R^c(T).
\end{align*}
This completes the proof.
\end{proof}

\begin{exam}
\begin{enumerate}
\item \label{exam:a}
Consider the planar binary tree $T=\YY{\yy10\yy21
\node[scale=0.8] at (0.75,-0.65) {$\alpha$};
\node[scale=0.8] at (0.05,-0.4) {$\beta$};
\node[scale=0.7] at (0.55,-0.2) {$\gamma$};
}$. By Theorem~\mref{them:comb} and Example~\mref{exam:comb2}~(\mref{it:a}), we have
\begin{align*}
\Delta_{LR(X)}(\YY{\yy10\yy21
\node[scale=0.8] at (0.75,-0.55) {$\alpha$};
\node[scale=0.8] at (0.05,-0.4) {$\beta$};
\node[scale=0.7] at (0.55,-0.2) {$\gamma$};
})= \YY{\yy10\yy21
\node[scale=0.8] at (0.75,-0.55) {$\alpha$};
\node[scale=0.8] at (0.05,-0.4) {$\beta$};
\node[scale=0.7] at (0.55,-0.2) {$\gamma$};
}\otimes |
+
|\otimes
\YY{\yy10\yy21
\node[scale=0.8] at (0.75,-0.55) {$\alpha$};
\node[scale=0.8] at (0.05,-0.4) {$\beta$};
\node[scale=0.7] at (0.55,-0.2) {$\gamma$};
}+
\YY{\yy11
\node[scale=0.8] at (0.25,-0.55) {$\beta$};
\node[scale=0.8] at (0.95,-0.30) {$\gamma$};
}\otimes\YY{\node[scale=0.8] at (0.75,-0.55) {$\alpha$};}+\YY{\node[scale=0.8] at (0.75,-0.55) {$\gamma$};}\otimes
\YY{\yy10
\node[scale=0.8] at (0.75,-0.55) {$\alpha$};
\node[scale=0.8] at (0.04,-0.3) {$\beta$};
}.
\end{align*}

\item \label{exam:b}
Let $T=\YY{\yy20\yy23
\node[scale=0.8] at (-0.1,-0.25) {$\beta$};
\node[scale=0.8] at (1.10,-0.25) {$\gamma$};
\node[scale=0.8] at (0.75,-0.55) {$\alpha$};
}$. It follows from Theorem~\mref{them:comb} and Example~\mref{exam:comb2}~(\mref{it:b}) that
\begin{align*}
\Delta_{LR(X)}(\YY{\yy20\yy23
\node[scale=0.8] at (-0.1,-0.25) {$\beta$};
\node[scale=0.8] at (1.10,-0.25) {$\gamma$};
\node[scale=0.8] at (0.75,-0.55) {$\alpha$};
})=&\YY{\yy20\yy23
\node[scale=0.8] at (-0.1,-0.25) {$\beta$};
\node[scale=0.8] at (1.10,-0.25) {$\gamma$};
\node[scale=0.8] at (0.75,-0.55) {$\alpha$};
}\otimes |
+|\otimes
\YY{\yy20\yy23
\node[scale=0.8] at (-0.1,-0.25) {$\beta$};
\node[scale=0.8] at (1.10,-0.25) {$\gamma$};
\node[scale=0.8] at (0.75,-0.55) {$\alpha$};
}+
(\YY{\node[scale=0.8] at (0.25,-0.55) {$\beta$};}\ast\YY{\node[scale=0.8] at (0.75,-0.55) {$\gamma$};})\otimes \YY{\node[scale=0.8] at (0.75,-0.55) {$\alpha$};}+\YY{\node[scale=0.8] at (0.25,-0.55) {$\beta$};}\otimes
\YY{\yy11
\node[scale=0.8] at (0.75,-0.55) {$\alpha$};
\node[scale=0.8] at (0.95,-0.30) {$\gamma$};
}+\YY{\node[scale=0.8] at (0.75,-0.55) {$\gamma$};}\otimes\YY{\yy10
\node[scale=0.8] at (0.75,-0.55) {$\alpha$};
\node[scale=0.8] at (0.04,-0.3) {$\beta$};
}\\
=&\YY{\yy20\yy23
\node[scale=0.8] at (-0.1,-0.25) {$\beta$};
\node[scale=0.8] at (1.10,-0.25) {$\gamma$};
\node[scale=0.8] at (0.75,-0.55) {$\alpha$};
}\otimes |+|\otimes
\YY{\yy20\yy23
\node[scale=0.8] at (-0.1,-0.25) {$\beta$};
\node[scale=0.8] at (1.10,-0.25) {$\gamma$};
\node[scale=0.8] at (0.75,-0.55) {$\alpha$};
}+
(\YY{\yy11
\node[scale=0.8] at (0.25,-0.55) {$\beta$};
\node[scale=0.8] at (0.95,-0.30) {$\gamma$};
}+
\YY{\yy10
\node[scale=0.8] at (0.75,-0.55) {$\gamma$};
\node[scale=0.8] at (0.04,-0.3) {$\beta$};
})\otimes \YY{\node[scale=0.8] at (0.75,-0.55) {$\alpha$};}+\YY{\node[scale=0.8] at (0.25,-0.55) {$\beta$};}\otimes
\YY{\yy11
\node[scale=0.8] at (0.75,-0.55) {$\alpha$};
\node[scale=0.8] at (0.95,-0.30) {$\gamma$};
}+\YY{\node[scale=0.8] at (0.75,-0.55) {$\gamma$};}\otimes\YY{\yy10
\node[scale=0.8] at (0.75,-0.55) {$\alpha$};
\node[scale=0.8] at (0.04,-0.3) {$\beta$};
}.
\end{align*}
\end{enumerate}
Observe that the results in~(\mref{exam:a}) and (\mref{exam:b}) are consistent with the corresponding ones in Example~\mref{ex:del}.
\end{exam}

As a direct consequence of Theorem~\mref{them:comb}, we may give another proof of the following result, which was obtained
in~\cite[Sec.~4.3]{Fois02} and~\cite[Sec.~6.3.5]{Man08}.

\begin{coro} For each $n\geq 0$,
$$\Delta_{LR(\X)}(\bfk Y_n(\X)) \subseteq \bigoplus_{p+q = n} \bfk Y_p(\X) \ot \bfk Y_q(\X).$$
\end{coro}

\begin{proof}
Let $T\in Y_n(\X)$. Denote by $\inte(T)$ the set of interior vertices of $T$. Then $|\inte(T)| = n$.
For an admissible cut $c$ in $\adm_*(T)$,
write
$$P^c(T) = T_1 \ast \cdots \ast T_k\,\text{ and }\, R^c(T) = T_{k+1}\,\text{ for some }\ k\geq 0.$$
Here we use the convention that $P^c(T) = |$ when $k=0$. By~\cite[Sec.~4.3]{Fois02}, the number of interior vertices of each summand in $T_1 \ast \cdots \ast T_k$ is  $\sum_{i=1}^k |\inte(T_i)|$.
So by Theorem~\mref{them:comb}, the number of interior vertices of each summand in $\Delta_{LR(\X)}(T)$ is
$$\sum_{i=1}^{k} |\inte(T_i)| +|\inte(T_{k+1})| = |\inte(T)| = n,$$
whence $$\Delta_{LR(\X)} \in \bigoplus_{p+q = n} \bfk Y_p(\X) \ot \bfk Y_q(\X),$$
as required.
\end{proof}

\begin{remark}
By Theorem~\mref{them:comb}, the fact that $H_{LR}(\X)$ is a connected graded bialgebra is obvious.
\end{remark}

\subsection{Subcoalgebras of coalgebra of decorated planar binary trees}
In this subsection, we only consider the aforementioned coalgebraic structure on decorated planar binary trees,
and show that $H_{LR}(\X)$ is a strictly graded coalgebra.

Let $C$ be a coalgebra. If there exists a family of $\bfk$-submodules $\{C^{(n)}\mid n\geq 0\}$ of $C$ such that
\begin{enumerate}
\item $C=\bigoplus_{n\geq 0}C^{(n)}$;
\vskip 0.1in
\item $\varepsilon (C^{(n)})=0$ , $n\neq 0$; and
\vskip 0.1in
\item $\Delta(C^{(n)})\subseteq \bigoplus_{p+q=n}C^{(p)}\ot C^{(q)}, n\geq 0$,
\end{enumerate}
then $C$ is called a {\bf graded coalgebra}.
If in particular, $$C^{(0)}\cong \bfk \text { and } C^{(1)}=P(C),$$
then $C$ is said to be a {\bf strictly graded coalgebra}~\cite[Chap.~4.1]{Eii80}, where $P(C)$ is the set of primitive elements of $C$.

\smallskip

\begin{defn}~\cite[Chap.~3.1]{Eii80}
Let $C$ be a coalgebra.
\begin{enumerate}
\item A subcoalgebra $M$ of $C$ is called a {\bf simple subcoalgebra} if it does not have any subcoalgebras other than $0$ and $M$.

 \item $C$ is called {\bf irreducible} if $C$ has only one simple subcoalgebra.

 \item $C$ is called {\bf pointed} if all simple subcoalgebras of $C$ are one dimensional.
 \end{enumerate}
\end{defn}

\smallskip

\begin{lemma}~\cite[Chap.~4.1]{Eii80}
A strictly graded coalgebra is a pointed irreducible coalgebra.
\mlabel{lem:sgc}
\end{lemma}

Narrowing our attention to the coalgebraic structure of $H_{LR}(\X)$, we obtain

\begin{theorem}
The coalgebra $(H_{LR}(\X), \Delta_{LR(\X)}, \varepsilon_{LR(\X)})$ is a strictly graded coalgebra with
the grading $H_{LR}(\X)= \oplus_{n\geq 0} \bfk Y_{n}(\X)$ and hence has only one simple subcoalgebra $\bfk \{|\}$.
\mlabel{them:sgc}
\end{theorem}

\begin{proof}
By Lemma~\mref{lem:hopf1}, $H_{LR}(\X)= \oplus_{n\geq 0} \bfk Y_{n}(\X)$ is a graded coalgebra. Since $Y_0(\X) = Y_0 = \{|\}$,
we have $\mathbf{k}Y_{0}(\X)=\mathbf{k}$.
Furthermore,
$$\mathbf{k}Y_{1}(\X)= \mathbf{k} \{\YY{\node[scale=0.8] at (0.75,-0.55) {${\alpha}$};}\mid \alpha\in \X\}$$
is the set of primitive elements of $H_{LR}(\X)$ by Theorem~\mref{them:comb}.
Thus $H_{LR}(\X)$ is a strictly graded coalgebra. By Lemma~\mref{lem:sgc}, $H_{LR}(\X)$ has only one simple subcoalgebra. Then the result follows from that $\mathbf{k}Y_{0}(\X)=\bfk \{|\}$ is a simple subcoalgebra of $H_{LR}(\X)$.
 \end{proof}

\begin{remark}
Summing up, the coradical of the Hopf algebra $H_{LR}(\X)$ (that is to say the sum of all its simple coalgebra) is $\mathbf{k}\{|\}$.
\end{remark}

\section{Free cocycle $\vee_\X$-Hopf algebras of decorated planar binary trees}
\mlabel{sec:fch}
In this section, based on the binary grafting operations $\vee_\alpha$ on $H_{LR}(\X)$ with $\alpha\in \X$,
we introduce the concept of a $\vee$-algebra and more generally a $\vee_\X$-algebra, leading to the emergence of $\vee_\X$-bialgebras
and $\vee_\X$-Hopf algebras. Further when a $\vee$-cocycle condition is involved, cocycle $\vee_\X$-bialgebras and cocycle $\vee_\X$-Hopf algebras are also introduced. We finally show that $H_{LR}(\X)$ is a free cocycle $\vee_\X$-Hopf algebra.

\subsection{Free $\vee_\X$-algebras of decorated planar binary trees}
In this subsection, we equip the space $H_{LR}(\X)$ of planar binary trees decorated by a nonempty set $\X$ with a
free $\vee_\X$-algebra structure. Let us first recall the concept of operated algebras.

\begin{defn}\cite[Sec.~1.2]{Guo09}
\begin{enumerate}
\item
An {\bf operated algebra } is an algebra $A$ together with a (linear) operator $P: A\to A$.
\item
An {\bf $\X$-operated algebra } is an algebra $A$ together with a set of (linear) operators $P_{\alpha}: A\to A$, $\alpha\in \X$.
\end{enumerate}
\end{defn}

Motivated by the above definition and Eq.~(\mref{eq:astttt}), we introduce $\vee$-algebras.

\begin{defn}
\begin{enumerate}
\item A {\bf $\vee$-algebra} is an algebra $(A, \ast_A, 1)$ together with a binary operation $\vee:A\ot  A \to A$ such that,
for $a=a_1\vee a_2$ and $a'=a_1'\vee a_2'$ in $A $,
\begin{align*}
 a\ast_A a'=a_1 \vee (a_2\ast_{A} a')+(a\ast_A a_1')\vee a_2'.
 %
\end{align*}
\end{enumerate}
More generally, let $\X$ be a nonempty set.
\begin{enumerate}\setcounter{enumi}{1}
\item A {\bf $\vee_\X$-algebra} is an algebra $(A, \ast_A, 1_A)$ equipped with a set of binary operations
$$\vee_\X := \{\vee_\alpha:A\ot  A \to A \mid \alpha\in \X \}$$ such that
\begin{align}
 a\ast_{A} a'=a_1 \vee_{\alpha} (a_2\ast_{A} a')+(a\ast_{A} a_1')\vee_{\alpha'} a_2',
\mlabel{eq:AA}
\end{align}
where $a=a_1 \vee_\alpha a_2$ and $a'=a_1'\vee_{\alpha'} a_2'$ in $A$ with $\alpha, \alpha'\in \X$. We denote such a $\vee_\X$-algebra by $(A, \ast_A, 1_A, \vee_\X)$.
\item
Let $(A, \ast_A, 1_A, \vee_\X)$ and $(A', \ast_{A'}, 1_{A'}, \vee'_\X)$ be two $\vee_\X$-algebras. A linear map $\phi : A\rightarrow A'$ is called a {\bf $\vee_\X$-algebra morphism} if $\phi$ is an algebra homomorphism such that $\phi \circ \vee_\alpha = \vee'_\alpha \circ (\phi\ot \phi)$ for each $\alpha\in \X$.

\item A {\bf free $\vee_\X$-algebra} on a set $X$ is a $\vee_\X$-algebra $(A, \ast_A, 1_A, \vee_\X)$ together with a set map $j:X\rightarrow A$ with the property that, for any $(A', \ast_{A'}, 1_{A'}, \vee'_\X)$ and a set map $\phi:X\rightarrow A'$, there exists a unique $\vee_\X$-algebra morphism $\bar{\phi}:A\rightarrow A'$
    such that $\bar{\phi} \circ j=\phi$.
\end{enumerate}
\mlabel{defn:veealg}
\end{defn}

\begin{remark}
Let us emphasize that $\vee$-algebras are different from operated algebras:
the $\vee$-algebra is an algebra equipped with a binary operation satisfying Eq.~(\mref{eq:AA}); but
the operated algebra is an algebra equipped with a unary operation.
\end{remark}

\begin{exam}
It follows from Eq.~(\mref{eq:astttt}) that $(H_{LR}(\X), \ast, |, \vee_\X)$ is a $\vee_\X$-algebra.
Indeed, it is a free $\vee_\X$-algebra with a universal property (see Theorem~\mref{thm:fva} below).
\end{exam}

The significant role of the binary grafting $\vee_{\X}$ is clarified by the following universal property.

\begin{theorem}
The quadruple $(H_{LR}(\X), \ast, |, \vee_{\X})$  is the free $\vee_{\X}$-algebra on the empty set,
that is, the initial object in the category of  $\vee_{\X}$-algebras.
More precisely, for any $\vee_\X$-algebra $A=(A, \ast_A, 1_A, \vee_{A, \X})$, there exists a unique $\vee_\X$-algebra morphism $\bar{\phi}: H_{LR}(\X)\rightarrow A$.
%
\mlabel{thm:fva}
\end{theorem}
\begin{proof}

(Uniqueness).
Suppose that $\bar{\phi}: H_{LR}(\X) \rightarrow A$ is a $\vee_\X$-algebra morphism.
We prove the uniqueness of $\bar{\phi}(T)$ for basis elements $T\in Y_{\infty}(\X)$ by induction on $\dep(T)\geq 0$.
For the initial step of $\dep(T)=0$, we have $T = |$ and $\bar{\phi}(T) = \bar{\phi}(|) = 1_A$.
For the induction step of $\dep(T)\geq 1$, we may write $T = T_1\vee_{\alpha} T_2$ for some $\alpha\in \X$ and then
\begin{align*}
\bar{\phi}(T) =&\, \bar{\phi}(T_1\vee_{\alpha} T_2) =  \bar{\phi}\circ \vee_\alpha(T_1,T_2) = \vee_{A,\alpha}\circ (\bar{\phi}\otimes\bar{\phi})(T_1, T_2) \\
 =& \vee_{A, \alpha}\bigl(\bar{\phi}(T_1)\otimes\bar{\phi}(T_2)\bigr )=\bar{\phi}(T_1)\vee_{A, \alpha}\bar{\phi}(T_2).
\end{align*}
Here $\bar{\phi}(T_1)$ and $\bar{\phi}(T_2)$ are determined uniquely by the induction hypothesis and so $\bar{\phi}(T)$ is unique.

(Existence). Define a linear map ${\bar{\phi}}: H_{LR}(\X) \rightarrow A$ recursively on depth $\dep(T)$ for $T\in  Y_{\infty}(\X)$
by assigning
\begin{align}
\bar{\phi}(|) :=1_A\,\text{ and }\,  \bar{\phi}(T) :=\bar{\phi}(T_1\vee_\alpha T_2):=\bar{\phi}(T_1)\vee_{A,\alpha}\bar{\phi}(T_2),
\mlabel{eq:mun}
\end{align}
where $T=T_1\vee_\alpha T_2$ for some $T_1, T_2\in Y_\infty(\X)$ and $\alpha\in \X$.
Then for any $T_1, T_2\in Y_{\infty}(\X)$, we have
\begin{align*}
\bar{\phi}\circ \vee_\alpha(T_1,T_2)=\bar{\phi}(T_1\vee_\alpha T_2)=\bar{\phi}(T_1)\vee_{A, \alpha}\bar{\phi}(T_2)=\vee_{A, \alpha}\bigl(\bar{\phi}(T_1)\otimes\bar{\phi}(T_2) \bigr) = \vee_{A, \alpha} \circ (\bar{\phi}\ot \bar{\phi})(T_1, T_2)
\end{align*}
and so
\begin{align*}
\bar{\phi}\circ \vee_\alpha =  \vee_{A, \alpha}\circ (\bar{\phi}\ot \bar{\phi}).
\end{align*}
We are left to check that
\begin{align}
\bar{\phi} (T\ast T')=\bar{\phi}(T)\ast_{A}\bar{\phi}(T')\,\text{ for any }\, T, T'\in Y_{\infty}(\X).
\mlabel{eq:hom}
\end{align}
We proceed to prove Eq.~(\mref{eq:hom}) by induction on the sum of depths $\dep(T)+ \dep(T')\geq 0$.
For the initial step of  $\dep(T)+ \dep(T') = 0$, we have $T = T'=|$ and
\begin{align*}
\bar{\phi} (T\ast T') =\bar{\phi}(|\ast |)=\bar{\phi}(|)=1_A=1_A\ast_{A} 1_A=\bar{\phi}(|)\ast_{A}\bar{\phi}(|) = \bar{\phi}(T) \ast_A \bar{\phi}(T').
 \end{align*}
For the induction step of $\dep(T)+ \dep(T') \geq 1$, if $\dep(T) = 0$ or $\dep(T') = 0$, without loss of generality, letting
$\dep(T) = 0$, then $T=|$ and
 \begin{align*}
\bar{\phi} (T\ast T') =\bar{\phi}(|\ast T')=\bar{\phi}(T')= 1_A \ast_A\bar{\phi}(T')= \bar{\phi}(|) \ast_A \bar{\phi}(T') = \bar{\phi}(T) \ast_A \bar{\phi}(T').
 \end{align*}
So we may assume that $\dep(T), \dep(T')\geq 1$ and write
\begin{align*}
T = T_1 \vee_\alpha T_2 \text{\ and\ } T' = T'_1 \vee_\beta T'_2\,\text{ for some }\, \alpha,\beta\in \X.
\end{align*}
Hence
\allowdisplaybreaks{
\begin{align*}
\bar{\phi}(T\ast T')&=\bar{\phi}\bigl(T_1\vee_\alpha (T_2\ast T')+(T\ast T_1')\vee_\beta T_2'\bigr)\quad  (\text{by Eq.~(\mref{eq:astttt})} )\\
&=\bar{\phi}\bigl(T_1\vee_\alpha (T_2\ast T')\bigr)+\bar{\phi}\bigl((T\ast T_1')\vee_\beta T_2'\bigr)\quad (\text{by $\bar{\phi}$ being linear}) \\
&=\bar{\phi}(T_1)\vee_{A,\alpha} \bar{\phi}(T_2\ast T')+\bar{\phi}(T\ast T_1')\vee_{A,\beta}\bar{\phi}(T_2') \quad (\text{by Eq.~(\mref{eq:mun})} )\\
&=\bar{\phi}(T_1)\vee_{A,\alpha}\bigl(\bar{\phi}(T_2)\ast_{A}\bar{\phi}( T')\bigr)+\bigl(\bar{\phi}(T)\ast_{A}\bar{\phi}( T_1')\bigr)\vee_{A,\beta} \bar{\phi}(T_2')\\
&\hspace{8cm} (\text{by\ the\ induction\ hypothesis})\\
&=\bar{\phi}(T_1)\vee_{A,\alpha}\Bigl(\bar{\phi}(T_2)\ast_{A} \bigl(\bar{\phi}( T_1')\vee_{A,\beta} \bar{\phi}(T_2')\bigr)\Bigr)+
\Bigl( \bigl(\bar{\phi}(T_1)\vee_{A,\alpha}\bar{\phi}(T_2) \bigr) \ast_{A}\bar{\phi}( T_1')\Bigr)\vee_{A,\beta} \bar{\phi}(T_2')\\
& \hspace{8cm}    \text{(by Eq.~(\mref{eq:mun}))}\\
&=\bigl(\bar{\phi}(T_1)\vee_{A,\alpha}\bar{\phi}(T_2)\bigr)\ast_{A}\bigl(\bar{\phi}( T_1')\vee_{A,\beta} \bar{\phi}(T_2')\bigr)\quad (\text{by\  Eq.~(\mref{eq:AA})} ) \\
&=\bar{\phi}(T_1\vee_{\alpha} T_2)\ast_{A}\bar{\phi}( T_1'\vee_{\beta} T_2')\quad (\text{by\  Eq.~(\mref{eq:mun})} )\\
&=\bar{\phi}(T)\ast_{A}\bar{\phi}(T'),
\end{align*}
}
as required.
This completes the proof.
\end{proof}

\subsection{Free cocycle $\vee_\X$-Hopf algebras of decorated planar binary trees}
In this subsection, we prove that $H_{LR}(\X)$ is the free cocycle $\vee_\X$-Hopf algebra on the empty set.
Let us first pose the following concepts which are motivated from the binary grafting operations $\{\vee_\alpha\mid \alpha\in \X\}$
characterized in Eq.~(\mref{eq:astttt}) and the coproduct $\Delta_{LR(\X)}$ given in Eq.~(\mref{eq:dlr22}).

\begin{defn}
\begin{enumerate}
\item A {\bf $\vee_\X$-bialgebra} (resp.~{\bf $\vee_\X$-Hopf algebra}) is a bialgebra (resp. Hopf algebra) $(H, \ast_H, 1_H, \Delta_H, \varepsilon_H )$ which is also a $\vee_\X$-algebra $(H,\ast_H, 1_H,\vee_{\X})$.

 \item Let $(H,\, \vee_{\X})$ and $(H',\,\vee'_{\X})$ be two $\vee_\X$-bialgebras (resp.~$\vee_\X$-Hopf algebras).
A linear map $\phi : H\rightarrow H'$ is called a {\bf $\vee_\X$-bialgebra morphism} (resp.~{\bf $\vee_\X$-Hopf algebra morphism})
if $\phi$ is a bialgebra (resp. Hopf algebra) morphism such that ${\phi}\circ\vee_{\alpha} = \vee'_{\alpha}\circ ({\phi}\ot {\phi})$ for $\alpha\in \X$.
\end{enumerate}
\mlabel{defn:veehopf0}
\end{defn}

Involved with an analogy of the Hochschild 1-cocycle condition~\mcite{Fo3}, we pose

\begin{defn}
\begin{enumerate}
\item An {\bf $\X$-cocycle $\vee_\X$-bialgebra} or simply a {\bf cocycle $\vee_\X$-bialgebra} is a $\vee_\X$-bialgebra $(H, \ast_H, 1_H, \Delta_H, \varepsilon_H, \vee_\X)$ satisfying the following {\bf $\vee$-cocycle condition:} for any $\alpha\in \X$ and $h,h'\in H,$
\begin{align}
\Delta_H (h \vee_\alpha h') = (h \vee_\alpha h') \otimes 1_H +(\ast_H, \vee_\alpha)\bigl(\Delta_H(h)\otimes \Delta_H(h')\bigr),
\mlabel{eq:cocycle}
\end{align}
where $(\ast_H,\vee_\alpha):=(\ast_H \otimes \vee_H)\circ {\tau_{23}}$ and $\tau_{23}$ is the permutation of the second and third tensor factor. If the bialgebra in a cocycle $\vee_\X$-bialgebra is a Hopf algebra, then it is called a {\bf cocycle $\vee_\X$-Hopf algebra}.

\item A {\bf free cocycle $\vee_\X$-bialgebra} on a set $X$ is a cocycle $\vee_\X$-bialgebra $(H, \ast_H, 1_H,\Delta_H,\varepsilon_H, \vee_\X)$ together with a set map $j: X\rightarrow H$ with the property that, for any cocycle $\vee_\X$-bialgebra $(H', \ast_{H'}, 1_{H'}, \Delta_{H'}, \varepsilon_{H'}, \vee'_{\X})$ and any set map $\phi: X \rightarrow H'$, there exists a unique $\vee_\X$-bialgebra morphism $\bar{\phi}:H \rightarrow H'$
    such that $\bar{\phi} \circ j =\phi$.
The concept of {\bf a free cocycle $\vee_\X$-Hopf algebra} is defined in the same way.
\end{enumerate}
\mlabel{defn:veehopf}
\end{defn}
When $\X$ is a singleton set, the subscript $\X$ in Definitions~\mref{defn:veealg} and~\mref{defn:veehopf} will be suppressed for simplicity.

The following result gives a family of coideals of a cocycle $\vee_\X$-bialgebra.
Recall that a submodule $I$ in a coalgebra $(C, \Delta, \varepsilon)$ is called a {\bf coideal}
if $I\subseteq \ker \varepsilon$ and  $\Delta(I) \subseteq I\ot C + C\ot I$.  A {\bf biideal} of a bialgebra $A$ is a submodule of $A$ which is both an ideal and a coideal of $A$.

\begin{prop}
Let $(H, \ast_{H}, 1_H, \Delta_{H}, \varepsilon_H, \vee_{\X})$
be a cocycle $\vee_\X$-bialgebra and $C$ a coideal of $H$. Then, we have the following.
\begin{enumerate}
\item \mlabel{it:ida}
$H\vee_{\alpha} H := \{h_1 \vee_\alpha h_2 \mid h_1, h_2\in H\}$ is a coideal of $H$ for each $\alpha\in \X$.

\item \mlabel{it:idb}
The ideal generated by $C$ is a biideal.
\end{enumerate}\mlabel{pp:coideal}
\end{prop}

\begin{proof}
(\mref{it:ida})
Let $\alpha\in \X$. We first show $H\vee_\alpha H\subseteq \ker \varepsilon_H$.
Let
$$h =h_1\vee_{\alpha}h_2 \in H\vee_\alpha H \,\text{ with }\, h_1, h_2 \in H.$$
Using Sweedler notation, we can write
\begin{align*}
\Delta_H(h_1)=\sum_{(h_1)}h_{1(1)}\otimes h_{1(2)}\, \text{ and }\, \Delta_H(h_2)=\sum_{ (h_2)}h_{2(1)}\otimes h_{2(2)}.
\end{align*}
Then
\allowdisplaybreaks{
\begin{align*}
h_1\vee_{\alpha}h_2=&\ (\varepsilon_H \ot \id)\circ \Delta_{H}(h_1\vee_{\alpha}h_2)\quad\quad\quad (\text{by the counicity})\\
=&\ (\varepsilon_H \ot \id)\Big((h_1 \vee_\alpha h_2) \otimes 1_H +(\ast_H, \vee_\alpha)\bigl(\Delta_H(h_1)\otimes \Delta_H(h_2)\bigr)\Big)\quad \text{(by\ Eq.~(\mref{eq:cocycle}))}\\
=&\ \varepsilon_H (h_1 \vee_\alpha h_2) \ot 1_H+(\varepsilon_H \ot \id)\circ(\ast_H, \vee_\alpha)\left(\bigg(\sum_{(h_1)}h_{1(1)}\otimes h_{1(2)}\bigg)\otimes \bigg(\sum_{ (h_2)}h_{2(1)}\otimes h_{2(2)}\bigg)\right)\\
=&\ \varepsilon_H (h_1 \vee_\alpha h_2) \ot 1_H+(\varepsilon_H \ot \id)\circ(\ast_H, \vee_\alpha)\left(\sum_{(h_1), (h_2)}h_{1(1)}\otimes h_{1(2)}\otimes h_{2(1)}\otimes h_{2(2)}\right)\\
=&\ \varepsilon_H (h_1 \vee_\alpha h_2) \ot 1_H+(\varepsilon_H \ot \id)\left( \sum_{(h_1), (h_2)}(h_{1(1)}\ast_{H} h_{2(1)})\otimes (h_{1(2)}\vee_{\alpha}h_{2(2)}) \right)\\
&\hspace{5cm} (\text{by\ } (\ast_H,\vee_\alpha):=(\ast_H \otimes \vee_\alpha)\circ {\tau_{23}})\\
=&\ \varepsilon_H (h_1 \vee_\alpha h_2) \ot 1_H+ \sum_{(h_1), (h_2)}\varepsilon_H(h_{1(1)}\ast_{H} h_{2(1)})\otimes (h_{1(2)}\vee_{\alpha}h_{2(2)}) \\
=&\ \varepsilon_H (h_1 \vee_\alpha h_2) \ot 1_H+ \sum_{(h_1), (h_2)}\Big(\varepsilon_H(h_{1(1)})\ast_{\bfk} \varepsilon_H(h_{2(1)})\Big) (h_{1(2)}\vee_{\alpha}h_{2(2)}) \\
=&\ \varepsilon_H (h_1 \vee_\alpha h_2) \ot 1_H+ \sum_{(h_1), (h_2)}\Big(\varepsilon_H(h_{1(1)})\ast_{\bfk} h_{1(2)}\Big) \vee_{\alpha}  \Big( \varepsilon_H(h_{2(1)}) \ast_{\bfk} h_{2(2)}\Big) \\
=&\ \varepsilon_H (h_1 \vee_\alpha h_2) \ot 1_H+h_1 \vee_{\alpha}  h_2\quad (\text{by the counicity}),
\end{align*}
}
which implies
$$\varepsilon_H (h_1 \vee_\alpha h_2)\ot 1_H=0\,\text{ and so }\, \varepsilon_H (h_1 \vee_\alpha h_2)=0.$$
We next show
$$H\vee_\alpha H\subseteq  (H\vee_\alpha H)\ot H + H\ot (H\vee_\alpha H).$$
Indeed, for any $h_1\vee_\alpha h_2\in H\vee_\alpha H$,
\begin{align*}
\Delta_{H}(h_1\vee_{\alpha}h_2)&=(h_1\vee_{\alpha}h_2)\otimes 1_{H}+(\ast_H, \vee_{\alpha})\bigl(\Delta_H(h_1)\otimes\Delta_H(h_2)) \quad \text{ (by Eq.~(\mref{eq:cocycle}))}\\
&= (h_1\vee_{\alpha}h_2)\otimes 1_{H}+(\ast_H, \vee_{\alpha}) \bigg( \sum_{(h_1), (h_2)}(h_{1(1)}\otimes h_{1(2)})\otimes (h_{2(1)}\otimes h_{2(2)}) \bigg)\\
&= (h_1\vee_{\alpha}h_2) \otimes 1_{H}+\sum_{(h_1), (h_2)}(h_{1(1)}\ast_H h_{2(1)})\otimes (h_{1(2)}\vee_{\alpha}h_{2(2)})\\
&\in (H\vee_\alpha H) \otimes H+ H\otimes (H\vee_\alpha H).
\end{align*}
Thus $H\vee_\alpha H$ is a coideal.

(\mref{it:idb}) Suppose that $I$ is the ideal generated by $C$. Then we can write $I=H\ast_{H}C\ast_{H}H$ and so
\begin{align*}
\Delta_{H}(I)&=\Delta_{H}(H\ast_{H}C\ast_{H}H) = \Delta_{H}(H)\ast_{H}\Delta_{H}(C)\ast_{H}\Delta_{H}(H)\\
&\subseteq (H\ot H)\ast_{H}(C\ot H+H\ot C)\ast_{H}(H\ot H)\\
&= (H\ast_{H}C\ast_{H}H) \ot (H \ast H\ast H) + (H\ast H\ast H) \ot (H\ast_{H}C\ast_{H}H)\\
&=I\ot H+H\ot I.
\end{align*}
Thus $I$ is a biideal of $H$.
\end{proof}

As a consequence of Proposition~\mref{pp:coideal}~(\mref{it:ida}), we obtain a family of coideals of $H_{LR}(\X)$.

\begin{coro}
The  $H_{LR}(\X) \vee_{\alpha} H_{LR}(\X)$ is a coideal of $H_{LR}(\X)$ for each $\alpha\in \X$.
\end{coro}

\begin{proof}
It follows from Proposition~\mref{pp:coideal}~(\mref{it:ida}).
\end{proof}

Now we are ready for our main result of this section.
\begin{theorem}
Let $\X$ be a nonempty set.
\begin{enumerate}
\item The sextuple $(H_{LR}(\X), \ast, |, \Delta_{LR(\X)}, \varepsilon_{LR(\X)}, \vee_\X)$
is the free cocycle $\vee_\X$-bialgebra on the empty set, that is, the initial object in the category of cocycle $\vee_\X$-bialgebras. \mlabel{it:concludeb}

\item The sextuple $(H_{LR}(\X), \ast, |, \Delta_{LR(\X)}, \varepsilon_{LR(\X)}, \vee_\X)$ is the free cocycle $\vee_\X$-Hopf algebra on the empty set,
, that is, the initial object in the category of cocycle $\vee_\X$-Hopf algebras. \mlabel{it:concludec}
\end{enumerate}
\mlabel{thm:conclude}
\end{theorem}

\begin{proof}

(\mref{it:concludeb})
It follows from Lemma~\mref{lem:hopf1} that $(H_{LR}(\X), \ast, |, \Delta_{LR(\X)}, \varepsilon_{LR(\X)})$ is a bialgebra.
Furthermore, the $(H_{LR}(\X), \ast, |, \Delta_{LR(\X)}, \varepsilon_{LR(\X)}, \vee_\X)$ is a $\vee_\X$-bialgebra by Eq.~(\mref{eq:astttt})
and a cocycle $\vee_\X$-bialgebra by Eq.~(\mref{eq:dlr22}).

We are left to show the freeness of $H_{LR}(\X)$. For this, let $(H, \ast_{H}, 1_H, \Delta_{H}, \varepsilon_H, \vee'_{\X})$
be an arbitrary cocycle $\vee_\X$-bialgebra.
In particular, $(H, \ast_H, 1_H, \vee'_\X)$ is a $\vee_\X$-algebra. So by Theorem~\mref{thm:fva},
there exists a unique algebra homomorphism $\bar{\phi}: H_{LR}(\X)\rightarrow H$ such that
\begin{equation}
\bar{\phi}\circ \vee_\alpha =  \vee'_\alpha \circ (\bar{\phi}\ot \bar{\phi}) \text{ for any } \alpha\in \X.
\mlabel{eq:phivee}
\end{equation}
It remains to check the following two points:
\begin{align}
\Delta_{H} \circ \bar{\phi} (T) &=(\bar{\phi} \otimes \bar{\phi}) \circ \Delta_{LR(\X)}(T), \mlabel{eq:phic}\\
\varepsilon_H \circ \bar{\phi}(T) &= \varepsilon_{LR(\X)}(T) \, \text{ for\ all } \, T\in Y_\infty(\X). \mlabel{eq:phiv}
\end{align}
We prove Eq.~(\mref{eq:phic}) by induction on $\dep(T)\geq 0$. For the initial step of $\mathrm{dep}(T)=0$, we have $T=|$ and
\begin{align*}
\Delta_{H} \circ \bar{\phi} (T) &= \Delta_{H} \circ \bar{\phi} (|)=\Delta_{H} (1_H)=1_H\otimes 1_H=\bar{\phi} (|)\otimes\bar{\phi} (|)\\
&=(\bar{\phi}\otimes\bar{\phi})(|\otimes |)= (\bar{\phi}\otimes\bar{\phi})\circ\Delta_{LR(\X)}(|) = (\bar{\phi}\otimes\bar{\phi})\circ\Delta_{LR(\X)}(T).
\end{align*}
For the induction step of $\dep(T)\geq 1$, we may write $T = T_1\vee_\alpha T_2$ for some $T_1, T_2\in Y_{\infty}(\X)$ and $\alpha\in \X$. Using the Sweedler notaion,
\begin{align}
\Delta_{LR(\X)}(T_1)=\sum_{(T_1)}T_{1(1)}\otimes T_{1(2)} \text{\ and \ } \Delta_{LR(\X)}(T_2)=\sum_{(T_2)}T_{2(1)}\otimes T_{2(2)}.
\mlabel{eq:dec.}
\end{align}
Then
\allowdisplaybreaks{
\begin{align*}
&\Delta_{H} \circ \bar{\phi}(T)\\
&=\Delta_{H} \circ \bar{\phi}(T_1\vee_\alpha T_2)=\Delta_{H} \bigl(\bar{\phi}(T_1)\vee'_{\alpha} \bar{\phi}(T_2)\bigr) \quad \text{(by\  Eq.~(\mref{eq:phivee}))}\\
&=\big(\bar{\phi}(T_1)\vee'_{\alpha} \bar{\phi}(T_2)\big)\otimes 1_{H}+(\ast_{H},\vee'_{\alpha})\Bigl(\Delta_{H}\bigl(\bar{\phi}(T_1)\bigr)\otimes\Delta_{H}\bigl(\bar{\phi}(T_2)\bigr)\Bigr)\quad \text{(by\  Eq.~(\mref{eq:cocycle}))}\\
&=\bar{\phi}(T_1\vee_\alpha T_2)\otimes 1_{H}+(\ast_{H},\vee'_{\alpha})\bigg(\Bigl((\bar{\phi}\otimes \bar{\phi})\Delta_{LR(\X)}(T_1)\Bigr)\otimes \Bigl((\bar{\phi}\otimes \bar{\phi})\Delta_{LR(\X)}(T_2)\Bigr)\bigg)\\
& \hspace{4cm}  \text{(by Eq.~(\mref{eq:phivee}) and the induction hypothesis)}\\
&=(\bar{\phi}\otimes \bar{\phi})\Bigl((T_1\vee_\alpha T_2)\otimes |\Bigr)+(\ast_{H},\vee'_{\alpha})\bigg(\Bigl((\bar{\phi}\otimes \bar{\phi})\Delta_{LR(\X)}(T_1)\Bigr)\otimes \Bigl((\bar{\phi}\otimes \bar{\phi})\Delta_{LR(\X)}(T_2)\Bigr)\bigg)\\
&=(\bar{\phi}\otimes \bar{\phi})\Bigl((T_1\vee_\alpha T_2)\otimes |\Bigr)+(\ast_{H},\vee'_{\alpha}) \bigg( \sum_{(T_1),(T_2)}\Bigl(\bar{\phi}(T_{1(1)})\otimes \bar{\phi}(T_{1(2)})\Bigr)\otimes \Bigl(\bar{\phi}(T_{2(1)})\otimes \bar{\phi}(T_{2(2)})\Bigr) \bigg)\\
& \hspace{9cm}    \text{(by Eq.~(\mref{eq:dec.}))}\\
&=(\bar{\phi}\otimes \bar{\phi})\Bigl((T_1\vee_\alpha T_2)\otimes |\Bigr)+\sum_{(T_1),(T_2)}\Bigl(\bar{\phi}(T_{1(1)})\ast_{H} \bar{\phi}(T_{2(1)})\Bigr)\otimes \Bigl(\bar{\phi}(T_{1(2)})\vee'_{\alpha} \bar{\phi}(T_{2(2)})\Bigr)\\
&=(\bar{\phi}\otimes \bar{\phi})\Bigl((T_1\vee_\alpha T_2)\otimes |\Bigr)+\sum_{(T_1),(T_2)}\bar{\phi}(T_{1(1)}\ast T_{2(1)})\otimes \bar{\phi}(T_{1(2)}\vee_\alpha T_{2(2)})\\
& \hspace{4cm} \text{(by $\bar{\phi}$ being an algebra homomorphism and Eq.~(\mref{eq:phivee}))}\\
&=(\bar{\phi}\otimes \bar{\phi})\Bigl((T_1\vee_\alpha T_2)\otimes |\Bigr)+(\bar{\phi}\otimes \bar{\phi}) \bigg( \sum_{(T_1),(T_2)}(T_{1(1)}\ast T_{2(1)})\otimes (T_{1(2)}\vee_\alpha T_{2(2)}) \bigg)\\
&=(\bar{\phi}\otimes \bar{\phi})\Bigl((T_1\vee_\alpha T_2)\otimes |\Bigr)+ (\bar{\phi}\otimes \bar{\phi}) \circ (\ast, \vee_\alpha)\Bigl(\Delta_{LR(\X)}(T_1)\otimes \Delta_{LR(\X)}(T_2)\Bigr)\\
&=(\bar{\phi}\otimes \bar{\phi})\biggl((T_1\vee_\alpha T_2)\otimes |+(\ast, \vee_\alpha)\Bigl(\Delta_{LR(\X)}(T_1)\otimes \Delta_{LR(\X)}(T_2)\Bigr)\biggr)\\
&=(\bar{\phi}\otimes \bar{\phi}) \circ\Delta_{LR(\X)}(T_1\vee_\alpha T_2)\\
&=(\bar{\phi}\otimes \bar{\phi}) \circ \Delta_{LR(\X)}(T).
\end{align*}
}
We next prove the Eq.~(\mref{eq:phiv}). If $T=|$, then
\begin{align*}
\varepsilon_H \circ \bar{\phi}(T) = \varepsilon_H \circ \bar{\phi}(|)=\varepsilon_H(1_H)=1_{\mathbf{k}}=\varepsilon_{LR(\X)}(|).
\end{align*}
If $T\neq |$, then $T$ can be written as $T = T_1\vee_\alpha T_2$ for some $T_1, T_2\in Y_{\infty}(\X)$ and $\alpha\in \X$. By Eq.~(\mref{eq:phivee}),
\begin{align}
\varepsilon_H \circ \bar{\phi}(T)=\varepsilon_H \circ \bar{\phi}(T_1\vee_\alpha T_2)=\varepsilon_H \bigl(\bar{\phi}(T_1)\vee'_\alpha  \bar{\phi}(T_2)\bigr)= 0 =\varepsilon_{LR(\X)}(T),
\mlabel{eq:count}
\end{align}
where the second last step employs Proposition~\mref{pp:coideal}~(\mref{it:ida}).
This completes the proof of Item~(\mref{it:concludeb}).

(\mref{it:concludec}) By Lemma~\mref{lem:hopf1}, $(H_{LR}(\X), \ast, |, \Delta_{LR(\X)}, \varepsilon_{LR(\X)})$ is a Hopf algebra.
It is further a $\vee_\X$-Hopf algebra by Eq.~(\mref{eq:astttt}) and a cocycle $\vee_\X$-Hopf algebra by Eq.~(\mref{eq:dlr22}).
Then Item~(\mref{it:concludec}) follows from Item~(\mref{it:concludeb}) and the well-known fact that
any bialgebra morphism between two Hopf algebras is compatible with the antipodes~\cite[Lem.~4.04]{Swe69}.
\end{proof}

Taking $\X$ to be a singleton set in Theorem~\mref{thm:conclude}, all planar binary trees in $H_{LR}(\X)$ are decorated by the same letter.
In other words, planar binary trees in $H_{LR}(\X)$ have no decorations in this case and that are precisely the
planar binary trees in the classical Loday-Ronco Hopf algebra $H_{LR}$. So

\begin{coro}
\begin{enumerate}
\item The classical Loday-Ronco Hopf algebra $H_{LR}$ is the free cocycle $\vee$-bialgebra on the empty set,
that is, the initial object in the category of cocycle $\vee$-bialgebras.

\item The classical Loday-Ronco Hopf algebra $H_{LR}$ is the free cocycle $\vee$-Hopf algebra on the empty set,
that is, the initial object in the category of cocycle $\vee$-Hopf algebras.
\end{enumerate}
\mlabel{cor:clr}
\end{coro}

\begin{proof}
It follows from Theorem~\mref{thm:conclude} by taking $\X$ to be a singleton set.
\end{proof}

\medskip

\noindent {\bf Acknowledgments}: This work was supported by the National Natural Science Foundation
of China (Grant No.\@ 11771191 and 11501267), Fundamental Research Funds for the Central
Universities (Grant No.\@ lzujbky-2017-162), the Natural Science Foundation of Gansu Province (Grant
No.\@ 17JR5RA175).

\medskip

\noindent  We thank Prof. Foissy for helpful discussion and Proposition~\mref{pp:coideal} is inspired by the email communication with him.

\end{document}